\documentclass[12pt,a4paper,reqno]{amsart}
\input amssymb.sty
\newcommand\Tref[1]{{Theorem~\ref{#1}}}
\newcommand\Cref[1]{{Corollary~\ref{#1}}}
\newcommand\Lref[1]{{Lemma~\ref{#1}}}

\newtheorem{thm}{Theorem}[section] 

\newtheorem{cor}[thm]{Corollary}

\newtheorem{defn}[thm]{Definition}

\newtheorem{exmpl}[thm]{Example}

\newtheorem{lem}[thm]{Lemma}
\newtheorem{prop}[thm]{Proposition}

\newtheorem{rem}[thm]{Remark}

\newcommand\Dref[1]{{Definition~\ref{#1}}}
\newcommand\Rref[1]{{Remark~\ref{#1}}}

\newcommand\M[1][d]{{\operatorname{M}_{#1}}} 
\def\({\left(}
\def\){\right)}

\def\bM{\overline{\mathcal M}}
\def\MG{\mathcal M}
\def\CI{\mathcal I}
\def\CJ{\mathcal J}
\def\sgn{\operatorname{sgn}}

\def\id{\operatorname{id}}

\def\a{{\alpha}}
\def\trmat{\operatorname{tr}_{\operatorname{mat}}}
\def\alphaj{\a_{\operatorname{mat}}}
\def\tra{\operatorname{tr}}

\def\aqpol{\operatorname{\alpha}_{\operatorname{pol}}^{\bar q}}

\def\lam{{\lambda}}
\def\la{{\lambda}}

\def\tor{\operatorname{tor}}
\def\Ann{\operatorname{Ann}}
\def\sub{{\,\subseteq\,}}
\newcommand{\set}[1]{{\left\{#1\right\}}}

\newcommand\eq[1]{{(\ref{#1})}}

\def\cha{\operatorname{char}}

\def\II{{I\!\!\,I}}
\def\III{{I\!\!\,I\!\!\,I}}

\def\co{{\,{:}\,}}
\def\divides{{\!\,{|}\!\,}}

\def\ra{{\rightarrow}}
\def\N {{\mathbb {N}}}

\def\Z {{\mathbb {Z}}}

\newcommand\suchthat{{\,:\ \,}}

\def\Q{{\mathbb Q}}
\newcommand\comp[3][\bullet]{{{#1}_{{\if1#2{}\else{#2}\fi}{\if#3K{}\else{(#3)}\fi}}}} %
\newcommand\assoc{{\CI_{\mathcal V}}} %

\newif\ifXY %
\XYtrue  %
\ifXY \usepackage{xy}\fi %
\ifXY \xyoption{matrix}\xyoption{arrow}\xyoption{curve} \fi

\def\Zcd{{Zariski closed}}

\def\tr{\operatorname{trace}}

\usepackage[usenames]{color}

\begin{document}

\title[Specht's problem over commutative
Noetherian rings] %
{Specht's problem for associative
affine algebras over commutative Noetherian rings}

\author{Alexei Belov-Kanel}
\author{Louis Rowen}
\author{Uzi Vishne} %

\address{Department of Mathematics, Bar-Ilan University, Ramat-Gan
52900,Israel} %
\email{\{belova, rowen, vishne\}@math.biu.ac.il}

 \thanks{This work was supported by the Israel Science Foundation (grant no. 1207/12).}

\begin{abstract}
In a series of papers \cite{BRV1}, \cite{BRV2}, \cite{BRV3} we
introduced full quivers and pseudo-quivers of representations of
algebras, and used them as tools in describing PI-varieties of
algebras. In this paper we apply them to obtain
 a complete proof of Belov's solution of
Specht's problem for affine algebras over an arbitrary Noetherian
ring. The inductive step relies on a theorem that enables one to
find a ``$\bar q$-characteristic coefficient-absorbing polynomial
in each T-ideal $\Gamma$,'' i.e., a non-identity of the
representable algebra~$A$ arising from $\Gamma$, whose ideal of
evaluations in $A$ is closed under multiplication by $\bar
q$-powers of the characteristic coefficients of matrices
corresponding to the generators of $A$, where $\bar q$ is a
suitably large power of the order of the base field. The passage to
an arbitrary Noetherian base ring $C$ involves localizing at finitely
many elements  a kind of $C$, and reducing to the field case by
a local-global principle.
\end{abstract}

\maketitle

\newcommand\LL[2]{{\stackrel{\mbox{#1}}{\mbox{#2}}}}
\newcommand\LLL[3]{\stackrel{\stackrel{\mbox{#1}}{\mbox{#2}}}{\mbox{#3}}}
\newcommand\LLLL[4]
{\stackrel {\stackrel{\mbox{#1}}{\mbox{#2}}}
{\stackrel{\mbox{#3}}{{\mbox{#4}}}} }
\newcommand\AR[1]{{\begin{matrix}#1\end{matrix}}}

\tableofcontents
\setcounter{tocdepth}{1} {\small \tableofcontents}

\section{Introduction}

Until \S\ref{nonassoc}, all algebras are presumed to be associative (not necessarily with
unit element), over a given commutative ring $C$ having unit
element 1. The free (associative) algebra is denoted by $C\{x\}$,
whose elements are called \textbf{polynomials}. The
\textbf{T-ideal} of a set of polynomials \textbf{in an algebra}
$A$ is the ideal generated by all substitutions of these
polynomials in $A$. For example, the set $\id(A)$ of polynomial
identities of an algebra~ $A$ is a T-ideal of $C\{ x\}$. A T-ideal
is \textbf{finitely based} if it is generated as a T-ideal by
finitely many polynomials. For example, when $A$ is a commutative
algebra over a field of characteristic 0, $\id(A)$ is finitely
based, by the single polynomial $[x_1,x_2]= x_1x_2-x_2x_1.$

Our objective in this paper is to complete the affirmative proof
of the affine case of Specht's problem, that any affine PI-algebra
over an arbitrary commutative Noetherian ring satisfies the ACC
(ascending chain condition) on T-ideals, or, equivalently, any
T-ideal is finitely based. In characteristic $0$ over fields, this
is the celebrated theorem of Kemer~\cite{Kem1}. When $\cha (F)>0$
there are non-affine counterexamples \cite{B1,G}, with a
straightforward exposition given in \cite{BRV4}, so the best one
could hope for is a positive result for affine PI-algebras. Kemer
\cite{Kem11} proved this result for affine PI-algebras over
infinite fields, and Belov extended the theorem to affine
PI-algebras over arbitrary commutative Noetherian rings, in his
second dissertation, with the main ideas given in \cite{B2}. We
give full details of the proof (over arbitrary commutative
Noetherian rings), cutting through combinatoric complications by
utilizing the full strength of the theory of full quivers as
expounded in \cite{BRV1},\cite{BRV2}, and~\cite{BRV3}. Actually,
working over arbitrary commutative Noetherian base rings raises
the question of the ACC for T-ideals of algebras that do not
satisfy a PI (because the coefficients of the identities need not
be invertible), but we still can obtain a positive solution in
Theorem~\ref{SpechtNoeth1}.

Note that there is no hope for such a result over a non-Noetherian
commutative base ring $C$, because of the following observation:

\begin{lem}\label{Tid} If $\mathcal I \triangleleft C$, then $\mathcal I A$ is
a T-ideal of $A.$ In particular, $\mathcal I C\{ x\} $ is a
T-ideal of~$C\{ x\}.$
\end{lem}
\begin{proof} Clearly $\mathcal I A \triangleleft A$, and is closed under endomorphisms.
\end{proof}
Consequently, any chain of ideals of $C$ gives rise to a corresponding chain of
T-ideals.

The positive solution to Specht's problem has structural
applications, extending Braun's Theorem on the nilpotence of the
radical of a relatively free algebra to the case where the base
ring is Noetherian, cf.~Theorem~\ref{Jnilp}.

 It might be instructive to indicate briefly where
our approach differs from Kemer's characteristic 0 approach. Kemer
first obtains his deep \textbf{Finite Dimensionality Theorem} that
any algebra is PI-equivalent to a finite dimensional algebra $A$.
Extending the base field, one may assume the base field $K$ is
algebraically closed, so Wedderburn's principal theorem enables
one to decompose $A = \bar A \oplus J$ as vector spaces, where $J$
is the radical of $A$ and $\bar A $ can be identified with the
algebra $A/J$. In two deep lemmas, exposed in
\cite[Section~4.4]{BR}, Kemer shows that the nilpotence index of
$J$ and the vector space dimension of $\bar A$ over $K$ can be
described as invariants in terms of evaluations of polynomials on
$A$, and then working combinatorically he shows that these
computational invariants can be used to prove his Finite
Dimensionality Theorem. Kemer's Finite Dimensionality Theorem
fails for algebras over finite fields. Thus, we need some other
technique, and we turn to the theory of quivers of representations
of algebras into matrices, which were described computationally in
\cite{BRV2} and \cite{BRV3}.

Recall \cite[pp.~28ff.]{BR} that an algebra $A$ over an integral
domain $C$ is {\bf representable} if it can be embedded as a
$C$-subalgebra of $\M[n](K)$ for a suitable faithful commutative $C$-algebra
$K\supset C$ (which can be much larger than $C$).
 In \cite{BRV2} we considered
the {\bf full quiver} of a representation of an associative
algebra over a field, and determined properties of full quivers by
means of a close examination of the structure of Zariski closed
algebras, studied in \cite{BRV1}.

The full quiver (or pseudo-quiver) is a directed, loopless graph
without cycles, in which vertices correspond to simple subalgebras
and edges to elements of the radical. A maximal subpath of this
graph is called a {\bf branch}. In place of Kemer's lemmas, we
utilize the combinatorics of the full quiver to compute the
invariants described above.

 Our affirmation of Specht's problem (in the affine case) is
divided into two stages: First we assume that the base ring $C$ is
a field $F$ of order $q$ (where $q$ could be infinity). Recall,
when $q< \infty$, that $q$ is a power of $p = \cha (F)< \infty,$
and the Frobenius map $a \mapsto a^q$ is an $F$-algebra
endomorphism.

 In the second stage, using ring-theoretic methods, we
reduce from the case of a general ring $C$ to the situation of the
first stage.

The main difficulty in this approach is to discern whether the
algebras we are working with actually are representable. When the
base ring is an infinite field $F$, Kemer~\cite{Kem1} proved that
any relatively free affine $F$-algebra is representable; this is
also treated in \cite{B2} for $F$ finite, but the proof is rather
difficult. Consequently, we plan to treat the representability
theorem in a separate paper. Although this decision enables us to
provide a quicker and more transparent proof of Specht's problem,
it forces us to consider T-ideals $\CI$ for which $C\{ x \}/\CI$
need not a priori be representable. Accordingly, we need some
method for ``carving out'' T-ideals $\CI$ for which $C\{ x \}/\CI$
is representable.

In \cite{BRV3}, a \textbf{trace-absorbing polynomial} for an
algebra $A$ is defined as a non-identity of $A$ whose T-ideal is
also an ideal of the algebra $\hat A$ obtained by taking $A$
together with the traces adjoined. The main result of \cite{BRV3}
was that such polynomials exist for relatively free algebras.
Explicitly, we proved the following:

\begin{itemize}
\item Trace Adjunction Theorem~(\cite[Theorem~5.16]{BRV3}). Any
branch of a basic full quiver of a relatively free algebra $A$
naturally gives rise to a trace-absorbing polynomial of $A$.
\end{itemize}

The Trace Adjunction Theorem provides a powerful inductive tool.
For example, as indicated in~\cite{BRV4}, it streamlines the proof
of the rationality of Hilbert series of relatively free algebras.

 In this paper we need to consider
more generally the
 \textbf{characteristic coefficients} of a matrix $a$, by which we mean
the coefficients of its characteristic polynomial $\la ^n + \sum
_{k=0}^{n-1} \a_k \lam ^k.$ The $k$-th \textbf{characteristic
coefficient} of $a$ is $\a _k.$ For example, the trace and
determinant are respectively the $(n-1)$-th and $0$-th
characteristic coefficients. In characteristic 0 one can recover
all the characteristic coefficients from the traces, which is why
we only dealt with traces in \cite{BRV3}. But here we need to
generalize the result to arbitrary characteristic coefficients.
Furthermore, since the multilinearization process cannot be
reversed over a finite field, we cannot prove theorems about
absorbing arbitrary characteristic coefficients in this case, but
must content ourselves with absorbing $\bar q$-powers of
characteristic coefficients, which we call $\bar
q$-\textbf{characteristic coefficients}, for $\bar q$ a suitable
power of $q = |F|$. (In fact, for technical reasons involving the
quiver, we need to use an idea of Drensky~\cite{D} and consider
\textbf{symmetrized} characteristic coefficients,
cf.~Definition~\ref{sym}.)

Thus, we generalize ``trace-absorbing polynomials'' to ``$\bar
q$-characteristic coefficient\-absorbing polynomials,''
cf.~Definition~\ref{absorp}, for the inductive step in the
solution of Specht's problem. We do not prove here that affine
PI-algebras are representable (one of the keystones of Kemer's
theory in characteristic 0); this involves a more intense study of
full quivers, which we leave for a later paper. Nevertheless, once
we have answered Specht's problem affirmatively for representable
relatively free, affine algebras in Theorem~\ref{Spechtfin}, the
passage to Noetherian base rings enables us to verify Specht's
problem for all affine PI-algebras in Theorem~\ref{SpechtNoeth},
and an elementary module-theoretic argument yields the result for
all varieties in Theorem~\ref{SpechtNoeth1}.

Our approach parallels \cite{BRV3}, but with an emphasis on
working inside the set of evaluations of a given non-identity $f$
of the algebra $A$. Although as formulated in
\cite[Theorem~5.16]{BRV3}, the Trace Adjunction Theorem enables us
to obtain characteristic coefficient-absorbing non-identities,
here we need to find a nonzero substitution inside~$f$. This is
done in Theorem~\ref{traceq2}, but at the cost of a considerably
more involved proof than that of \cite[Theorem~5.12]{BRV3}. For
starters, in characteristic $p$, the multilinearization procedure
degenerates in the sense that one cannot recover a polynomial from
specializations of its multilinearizations. This means that one
could have a proper inclusion of T-ideals which contain exactly
the same multilinear polynomials (seen for example by taking the
Boolean identity $x^2 +x$ in characteristic $2$, whose
multilinearization is just the identity of commutativity), so the
inductive step in Specht's problem requires coping with
$A$-quasi-linear (and $A$-homogeneous) polynomials rather than just with
multilinear polynomials. Ironically, working in characteristic $p$
does yield one step that is easier, given in \Lref{barq}.

Recall that the trace-absorbing polynomial of
\cite[Theorem~5.12]{BRV3} is obtained by means of a ``hiking
procedure'' which forces substitutions of the polynomial into the
radical. The main innovation needed here in hiking arises from the
necessity to deal with several monomials of our polynomial $f$ at
a time, which was not the case in \cite{BRV3}. Thus, we introduce
hiking of ``higher stages,'' in particular \textbf{stage 2}
hiking, which eliminates substitutions of $f$ in the ``wrong''
matrix components, \textbf{stage 3} hiking, which differentiates
the sizes of the base fields of the different components of
maximal matrix degree, and \textbf{stage 4} hiking, which removes
hidden radical substitutions.

Two definitions of actions by characteristic coefficients (one in
terms of matrix computations and one in terms of polynomial
evaluations) can be defined on the T-ideal that is generated by
this polynomial, which thus is a common ideal of $A$ and the
algebra~$\hat A$ obtained by adjoining traces to $A$, and $\hat A$
is Noetherian by Shirshov's Theorem \cite[Chapter 2]{BR}. We
perform the same reasoning for $\bar q$-characteristic
coefficient-absorbing polynomials, but also need {stage 4} hiking
in \Lref{complem}, in order to identify these two actions.
This enables us to pass to Noetherian algebras and conclude the
verification of Specht's problem for algebras over arbitrary
fields, in Theorem~\ref{Spechtfin}.

The extension to algebras over an arbitrary commutative Noetherian
base ring $C$ is given in Theorems~\ref{SpechtNoeth} and
~\ref{SpechtNoeth1}. The proof has a different flavor, based on
considerations about $C$-torsion which lead to a formal reduction
to the case that $C$ is an integral domain, in which case we repeatedly apply a version of a local-global principle and
conclude by passing to its field of fractions and applying the
results from the previous paragraph.

 \section{Preliminary material}\label{backg}

Let us start by reviewing the background, especially about full
quivers, their relationship with relatively free algebras, and the
polynomials that they yield.

\subsection{Characteristic coefficients of matrices}

We start with some observations about characteristic coefficients
of matrices, which we need to utilize in characteristic $p$.

Any matrix $a \in \M[n](K)$ can be viewed either as a linear
transformation on the $n$-dimensional space $V = K^{(n)}$, and
thus having Hamilton-Cayley polynomial $f_a$ of degree~$n$, or
(via left multiplication) as a linear transformation $\tilde a$ on
the $n^2$-dimensional space $\tilde V = \M[n](K)$ with
Hamilton-Cayley polynomial $f_{\tilde a}$ of degree $n^2$.

\begin{rem}\label{eig}
The matrix $\tilde a$ can be identified with the matrix $$a
\otimes I \in \M[n](K) \otimes \M[n](K) \cong \M[n^2](K),$$ so its
eigenvalues have the form $\beta \otimes 1 = \beta$ for each
eigenvalue $\beta$ of $a$.
\end{rem}

\begin{lem}\label{noZub0}
Notation as above, $f_{\tilde a}= f_a^n$, over any integral domain
of arbitrary characteristic.
\end{lem}
\begin{proof} By a standard specialization argument, it is enough to check the equality over
the free commutative ring $\mathbb Z [ \xi_1, \xi_2, \dots],$
which can be embedded into an algebraically closed field of
characteristic 0. By Zariski density, we may assume that $a$ is
diagonal, in which case it is clear that the determinant of
$\tilde a$ is $\det(a)^n.$ But then we conclude by taking $\la ^n
-a$ instead of $a$.
\end{proof}
\Lref{noZub0} often is used in conjunction with the next
observation.

\begin{lem}\label{obv1} Suppose $a \in \M[n](F),$ with $\cha(F) = p$, and $f_a = |\la I - a| = \sum \a _i \la ^i$ is the
characteristic polynomial of $a$. Then, for any $p$-power $\bar
q,$ $\sum \a _i ^{\bar q} \la ^i$ is the characteristic polynomial
of $a^{\bar q}$.
\end{lem}
\begin{proof}
Follows from $f_{a^p}(\lam^p) = |\lam^p I - a^p| = |\lam I - a|^p
= f_{a}(\lam)^p$.
\end{proof}

\begin{prop}\label{obv2} Suppose $a \in \M[n](F).$ Then the
characteristic coefficients of $a$ are integral over the $F$-algebra $C$ generated by the
characteristic coefficients of $\tilde a$.
\end{prop}
\begin{proof} The integral closure $\bar C$ of $C$ contains all the eigenvalues of $\tilde a,$ which are the eigenvalues of $a,$ so the characteristic coefficients of $\tilde a$ also belong to $\bar C.$
\end{proof}

\begin{defn}\label{cc1} The
 $\a _i ^{\bar q}$ of \Lref{obv1} are called the $\bar
q$-\textbf{characteristic coefficients} of $a$.
\end{defn}

\begin{rem}\label{qwhat} We choose $\bar q$ sufficiently large so that the theory will run smoothly.
By \cite[Remark~2.35 and Lemma~2.36]{BR}, when $\bar q
>n$ (which is greater than the nilpotence index in the Jordan decomposition), then
the matrix $a^{\bar q}$ is semisimple. \end{rem}

\subsection{Varieties of PI-algebras}

We work with polynomials in the free algebra built from a
countable set of indeterminates over the given commutative
Noetherian base ring $C$. The set $\id(A)$ is well-known to be a
T-ideal of the free algebra $C\set{x}$. More generally, given a
polynomial $f$, we define $\langle f(A) \rangle$ to be the ideal
generated by $A$. Thus, $\langle f(A) \rangle = 0$ iff $f\in
\id(A)$.

A polynomial is \textbf{blended} if each indeterminate appearing
nontrivially in the polynomial appears in each of its monomials.
As noted in \cite[Remark~22.18]{BRV3}, any T-ideal is additively
spanned by T-ideals of blended polynomials, and we only consider
blended polynomials throughout this paper.

Given a T-ideal $\CI$ of the free algebra $C\{ x\},$ we can form
the \textbf{relatively free algebra} $C\{ x\}/\CI,$ which is free
in the class of all PI-algebras $A$ for which $\id (A) \subseteq
\CI.$ Using this correspondence, it is enough to classify
relatively free algebras.

We continue by taking our base ring to be a field $F$, and we
investigate relatively free PI algebras in terms of the full
quivers of their representations, making use of \textbf{generic
elements}, as constructed in \cite[Construction 7.14]{BRV1} and
studied in \cite[Theorem 7.15]{BRV1}. (A generic element of a
finite dimensional algebra having base $\{ b_1, \dots, b_n \}$
over an infinite field is just an element of the form $\sum \xi_i
b_i,$ where the $ \xi_i$ are indeterminates, but the situation for
algebras over a finite field becomes considerably more intricate.)

As in \cite{BRV3}, we rely heavily on the {\bf Capelli polynomial}
$$c_k(x_1, \dots, x_k; y_1, \dots, y_k)= \sum_{\pi \in S_k} \sgn (\pi) x_{\pi (1)}y_1 \cdots x_{\pi (k)} y_k$$
of degree $2k$ (cf.~\cite{BR}). Any $C$-subalgebra of $\M[n](K)$
satisfies the identities $c_k$ for all $k > n^2$.

Recall that a polynomial $f$ which is linear in the first $t$ variables is $t$-\textbf{alternating} if substituting $x_j \mapsto x_i$ results in $0$ for any $1\leq i < j \leq t$.
\begin{defn}\label{cent0}
We denote by $h_n$ the $n^2$-alternating central polynomial on
$n\times n$ matrices \cite[p.~25]{BR}. (We define formally $h_0 =
1$, and also have $h_1 = x_1$ and $h_2 = c_4 g$ where $g$ is the
multilinearization of the central polynomial $[y_1,y_2]^2$ for $2
\times 2$ matrices, where we use fresh indeterminates for $c_4$.)

When appropriate, we write $h_n(x)$ to emphasize that $h_n$ is
evaluated on indeterminates $x_1, x_2, \dots.$
\end{defn}
The central polynomial $h_n$ is a crucial polynomial for our
deliberations, since it is always central for $\M[n](C)$
regardless of the commutative base ring $C$.

\subsection{Quasi-linear functions and quasi-linearizations}

Although the theory works most smoothly for multilinear
polynomials, in characteristic $p$ we do not have the luxury of
being able to recover a (blended) polynomial from its
multilinearization, the way we can in characteristic 0. For
example, one cannot recover the Boolean identity $x^2-x$ from its
multilinearization $x_1x_2+x_2x_1$, which holds in any commutative
algebra of characteristic 2. Thus, we must stop the linearization
process before arriving at multilinear identities.

 It is convenient at times to
work slightly more generally with functions rather than
polynomials, in order to be able to apply linear transformations.

\begin{defn}
A function $f$ is $i$-\textbf{quasi-linear} on $A$ if
$$f(\dots, a_i + a_i', \dots) = f(\dots, a_i , \dots)+f(\dots,
a_i', \dots)$$ for all $a_i, a_i' \in A; $ $f$ is
$A$-\textbf{quasi-linear} if $f$ is $i$-quasi-linear on $A$ for
all $i$. \end{defn}

Quasi-linear polynomials are used heavily by Kemer in
\cite{Kem11}.
 In contrast to \cite{BRV3} in which quasi-linear
polynomials played a somewhat secondary role, here they are at the
forefront of the theory, since we cannot avoid finite fields.
Accordingly, we need to develop them here, expanding on
\cite[Exercise 1.9 and~1.10]{BR}.

\begin{rem} When $\cha(F) = p,$
any $p^t$ power of an $A$-quasi-linear central polynomial is
$A$-quasi-linear.
\end{rem}

Any identity of $A$ is obviously $A$-quasi-linear, since the only
values are 0, so quasi-linear polynomials are only interesting for
non-identities.

\begin{lem}
If $f$ is $A$-quasi-linear in $x_i$, then $$f(a_1, a_2, \dots,
a_{i,1}+\cdots+ a_{i,d_i}, \dots)= \sum _{j=1}^{d_i} f(a_1, a_2,
\dots, a_{i,j}, \dots), \quad \forall a_i \in A.$$\end{lem}
\begin{proof} The assertion is immediate from the
definition.\end{proof}

As usual, for any monomial $h(x_1, x_2, \dots ), $ define $\deg _i
h$ to be the degree of $h$ in $x_i$. For any polynomial $f(x_1,
x_2, \dots ) , $ define $\deg _i f$ to be the maximal degree $\deg
_i h$ of its monomials; the sum of all such monomials is called
the \textbf{leading $i$-part} of $f$.

\begin{defn} Suppose $f(x_1, x_2, \dots) \in C\{ x\}$ has degree $d_i$ in
$x_i$. The \textbf{$i$-partial linearization} of $f $ is
\begin{equation}\label{partlin}\Delta_i f := f(x_1, x_2, \dots, x_{i,1}+\cdots+ x_{i,d_i}, \dots)- \sum
_{j=1}^{d_i} f(x_1, x_2, \dots, x_{i,j}, \dots)\end{equation}
where the substitutions were made in the $i$ component, and
$x_{1,1},\dots,x_{1,d_i}$ are new variables.
\end{defn}

Note that the $i$-partial linearization procedure lowers the
degree of the polynomial in the various indeterminates, in the
sense that the degree in each $x_{i,j}$ is less than~ $d_i$. It
follows at once that applying the $i$-partial linearization
procedure repeatedly, if necessary, to each $x_i$ in turn in any
polynomial $f$, yields a polynomial that is $A$-quasi-linear.

In \cite{Kem11} and \cite{BRV3}, quasi-linearizations had been
defined slightly differently, as homogeneous components of partial
linearizations as defined in \eqref{partlin}; when $f$ belongs to
a variety $V$ defined over an infinite field, these remain in $V$.
However, in \cite[Example~2.2]{BRV4} we saw that over a finite
field these homogeneous components might not necessarily stay in
the same variety, which is the reason we have modified the
customary definition.

Formally, this procedure is slightly stronger than that given in
\cite{BRV3}, but yields the following nice result:

\begin{prop}\label{targil} Suppose
 $\cha(F) = p$ and $d_i = \deg _i f$ is not a $p$-power. Then the leading $i$-part of $f$ can be recovered from
a suitable specialization of the leading $k$-part of an $i$-partial
linearization of $f $, for suitable $k$.
\end{prop}
\begin{proof} Taking $k$ such that $\binom{d_i}k$ is not divisible
by $p,$ we note that \eqref{partlin} has $\binom{d_i}k$ terms of
degree $k$ in $x_{i,1}$ and degree $d_i -k$ in $x_{i,2}$, so we
specialize $x_{i,1}\mapsto x_i$, $x_{i,2}\mapsto 1,$ and all other
$x_{i,j}\mapsto 0.$
\end{proof}

\begin{cor} For any polynomial $f$ which is not an identity of $A$, the T-ideal generated by $f$
contains an $A$-quasi-linear non-identity for which the degree in
each indeterminate is a $p$-power, where $p = \cha (F).$ \end{cor}
\begin{proof} Apply Proposition~\ref{targil} repeatedly, until the
degree in each indeterminate is a $p$-power.
\end{proof}

\subsection{Radical and semisimple substitutions}

\begin{rem}\label{delpt} When studying a representation $\rho \co A \to M_n(K)$ of an algebra $A$,
we usually identify $A$ with its image. In case $A$ is an algebra
over a field, we write $A = S \oplus J$, the Wedderburn
decomposition into the semisimple part $S$ and the radical $J$.
Then we can choose the representation such that the Zariski
closure of $A$ has the \textbf{Wedderburn block decomposition} of
\cite[Theorem~5.7]{BRV1}, in which the semisimple part $S$ is
written as matrix blocks along the diagonal. 

A {\bf semisimple substitution} (into a \Zcd\ algebra $A$) is a
substitution into an element of $S$ in some Wedderburn block of
$A$, and a {\bf radical substitution} is a substitution into an
element of $J$ in some Wedderburn block. A {\bf pure substitution}
is a substitution that is either a semisimple substitution or a
radical substitution, i.e., into $S \oplus J$.
\end{rem}

\begin{defn} Write $A = S \oplus J$, the Wedderburn decomposition into the
semisimple part $S$ and the radical $J$. A {\bf semisimple
substitution} (into a \Zcd\ algebra $A$) is a substitution into an
element of $S$ in some Wedderburn block of $A$, and a {\bf radical
substitution} is a substitution into an element of $J$ in some
Wedderburn block. A {\bf pure substitution} is a substitution that
is either a semisimple substitution or a radical substitution,
i.e., into $S \oplus J$.
\end{defn}

\begin{rem}\label{sub1} By \cite[Remark~2.20]{BRV3}, one can check whether an
$A$-quasi-linear polynomial $f(x)$ is a PI of $A$ merely by
specializing the indeterminates $x_i$ to pure substitutions.

More generally, let $f$ be any polynomial. Given a substitution
$f(\overline{x_1}; \overline{x_2},\dots)$, if we specialize
$\overline{x_1} \mapsto \overline{x_{1,1}} +\overline{x_{1,2}} $,
then
\begin{equation}\label{quasi1}
f(\overline{x_1}; \overline{x_2}, \dots) = f(\overline{x_{1,1}} ; \overline{x_2},
\dots)+ f(\overline{x_{1,2}} ; \overline{x_2}, \dots)+ \Delta
f(\overline{x_{1,1}} ,\overline{x_{1,2}} ; \overline{x_2},
\dots),\end{equation} where $\Delta f$ is obtained from the
$1$-partial linearization by specializing $x_{1,j}\mapsto 0$ for
all $j>2$ and then discarding these $x_{1,j}$ from the notation.
\end{rem} One can interpret Equation~\eqref{quasi1} as follows:
\begin{lem}\label{sub2}
Suppose $x_1$ has some specialization $x_1 \mapsto \sum
\overline{x_{1,j}}$ where the $\overline{x_{1,j}}$ are pure
substitutions. (For example, some of them might be semisimple and
others radical.) Then all specializations involving ``mixing'' the
$\overline{x_{1,j}}$ occur in $ \Delta f(\overline{x_{1,1}}
,\overline{x_{1,2}} , \overline{x_2}, \dots)$.
\end{lem}
\begin{proof} The ``mixed'' substitutions do not occur in the first two
terms of the right side of Equation~\eqref{quasi1}.\end{proof}

\Lref{sub2} enables us to apply the quasi-linearization
procedure on specific substitutions of $A$, rather than on all of
$A$, and will be needed when studying specific specializations of
a polynomial $f$. If $f$ were linear in $x_1$ then we could
separate these into distinct specializations of $f$. But when $f$
is non-linear in $x_1$, we often need to turn to \Lref{sub2}.

In \cite{Kem11}, the definition of quasi-linear also included
homogeneity, which can be obtained automatically over infinite
fields. Here again, since we are working over finite fields, we
need to be careful. We say that a function $f$ is
$i$-\textbf{quasi-homogeneous} of degree~$s_i$ on~$A$ if
$$f(\dots, \a a_i, \dots) = \a ^ {s_i} f(\dots, a_i , \dots)$$ for all $\a \in F, a_i \in A; $ $f(x_1, \dots,
x_t; y_1, \dots, y_m)$ is $A$-\textbf{quasi-homogeneous} of degree
$s$ on~$A$, if $f$ is $i$-quasi-homogeneous on $A$ of degree $s_i$
for all $1 \le i\le t,$ with $s = s_1 \cdots s_t$.

The next lemma shows the philosophy of our approach, although we
cannot use it directly because we are working over finite fields.

\begin{rem}\label{quasi0} Suppose $f = \sum f_j \in \CI$, where each $f_j$ is
$i$-quasi-homogeneous of degree $ s_{i,j}$ on $A$. Then fixing
some $j_0$ and taking $s= s_{i,j_0}$ yields
$$f(\dots, \a a_i, \dots) - \a ^ s f(\dots, a_i , \dots) = \sum
_{j \ne j_0} (\a ^ {s_{i,j}} - \a ^ s)f_j(\dots, a_i , \dots).$$
This lowers the number of $i$-homogeneous components of $f$, and
provides an inductive procedure for reducing to quasi-homogeneous
functions.\end{rem}

\begin{lem}\label{quasi} Given any T-ideal $\CI$
and any polynomial $f\in \CI$ which is a non-identity of $A$, we
can obtain an $A$-quasi-homogeneous polynomial in $\CI$.
\end{lem}
\begin{proof} By Remark~\ref{quasi0}, taking $s$ to be the degree of $x_i$ in some monomial,
this monomial cancels in
$$f(\dots, \a a_i, \dots) - \a ^ s
f(\dots, a_i , \dots),$$ so one concludes by induction.
\end{proof}

\begin{defn}
A specialization \textbf{radically annihilates} the polynomial $f$
if the number of radical substitutions is at least the nilpotence
index of $J$.
\end{defn}

In case a substitution radically annihilates $f$, each monomial of
$f$ must evaluate to~0. One main idea here is that the nilpotence
of the radical forces an evaluation to be 0 when the
specialization radically annihilates the polynomial $f$.

At the outset, for full quivers defined over a field, the
semisimple part $S$ is the sum of the diagonal Wedderburn blocks
of $A$, and $J$ is the sum of the off-diagonal Wedderburn blocks.
However, after ``gluing up to infinitesimals,'' some of the
radical $J$ might be transferred to the diagonal blocks. For example,
when $A$ is a local algebra, there is a single block, which thus
contains all of $J$.

\begin{defn}\label{intern} A radical substitution is \textbf{internal} if it occurs in a
diagonal block (after ``gluing up to infinitesimals'');
otherwise it is \textbf{external}.
\end{defn}

\subsubsection{The Hamilton-Cayley equation applied to quasi-linear
polynomials}

One of the key techniques used here (and throughout combinatorial
PI-theory) is to absorb characteristic coefficients into some
(multilinear) alternating polynomial $f(x_1, \dots, x_t; y_1,
\dots, y_t),$ as exemplified in \cite[Theorem~J, p.~25]{BR}. Since
we must cope with quasi-linear polynomials in this paper, we need
to extend the theory to quasi-linear polynomials.
Accordingly, we need another definition.

\begin{defn} A polynomial $f(x_1, \dots,
x_t; y_1, \dots, y_t),$ is \textbf{$(A;t;\bar
q)$-quasi-alternating} if $f$ is $A$-quasi-linear in $x_1, \dots,
x_t$ and quasi-homogeneous of degree~$\bar q$, for which $f$
becomes 0 whenever $x_i$ is substituted throughout for $x_j$ for
any $1 \le i< j \le t.$\end{defn}

Fortunately, the task of working with quasi-linear polynomials
over infinite fields was already done in Kemer's verification of
\cite[Equation~(40)]{Kem11}; he uses the terminology
\textbf{forms} for our characteristic coefficients. If $f$ is
$(A;t;\bar q)$-quasi-alternating, then we still get Kemer's
conclusion. This can also be stated in the language of
\cite[Theorem~J, Equation~1.19, page 27]{BR} (with the same
proof), as follows:
\begin{equation}\label{traceab0}
\alpha _k ^{\bar q}f(a_1, \dots, a_t, r_1, \dots, r_m) = \sum
f(T^{k_1}a_1, \dots, T^{k_t}a_t, r_1, \dots, r_m),\end{equation}
summed over all vectors $(k_1, \dots, k_t)$ with each $k_i \in \{
0, 1\}$ and $k_1 + \dots + k_t = k,$ where $\alpha _k $ is the
$k$-th characteristic coefficient of a linear transformation $T \co V
\to V,$ and $f$ is $(A;t;\bar q)$-quasi-alternating. Of course,
when applying \eqref{traceab0} in arbitrary characteristic, we
must consider all characteristic coefficients and not just the
traces.

We want to determine a value of $\bar q$ for which our polynomial
$f$ will be $(A;t;\bar q)$-quasi-alternating. When dealing with a
representable affine algebra $A$, which has a finite number of
generators, we may assume that the base field $F$ is finite, and
thus any element, viewed as a matrix in $\M[n](K)$, must have all
of its characteristic values in a field $\bar F$ which is a field
extension of $F$ of some finite order $\bar q.$ The idea is to
take the characteristic polynomial of the matrix $a^{\bar q}$
instead of the characteristic polynomial of the matrix $a$.

\begin{rem}\label{noZub}
There is a delicate issue here, insofar as Amitsur's proof of
\cite[Theorem~J]{BR} relies on $T$ acting on a vector space $V$.
If we take $V = \M[n](K)$ then its dimension is $n^2$, but we can
bypass this difficulty by appealing to the upcoming
\Lref{obv3}. (Note also that when $n$ is a power of $p$, then
$n^2$ is still a power of $p$. In view of \Lref{noZub0}, we
can just replace $\bar q$ by $\bar q^2$. Likewise, we could
replace $T$ by any p-power of $T$.)
\end{rem}

\begin{lem}\label{obv3}
Suppose $C$ is an algebra containing the $\bar q$-characteristic
coefficients of a matrix $a \in \M[n](C)$. Then $a$ is integral
over $C$.\end{lem}
\begin{proof} By assumption $a^{\bar q}$ is integral over $C$, implying at once that $a$ is
integral over $C$.
\end{proof}

\begin{rem}\label{Shirsh}
For our applications of Shirshov's Theorem we only need to adjoin
finitely many ${\bar q}$-characteristic coefficients to a given
affine $C$-algebra $A$ to obtain an algebra $\hat A$ integral of
bounded degree over the commutative algebra $\hat C_{\bar q}$
obtained by adjoining the same ${\bar q}$-characteristic
coefficients to~$C$.

Thus, when we are given a representation $\rho \co A \to
\M[n](C)$, we stipulate that the generators of $\hat C_{\bar q}$
include all ${\bar q}$-characteristic coefficients of products of
a given finite set of generators of $\rho(A)$ (viewed as a matrix
in $\M[n](C)$).
\end{rem}

\begin{defn}
We call $\hat C_{\bar q}$ of Remark~\ref{Shirsh} the ${\bar
q}$-\textbf{characteristic closure} of $C$.
\end{defn}

\begin{lem}\label{ShirshL}
$\hat C_{\bar q} $ is its own ${\bar q}$-characteristic
closure.
\end{lem}
\begin{proof} We appeal to a result of Amitsur~\cite[Theorem~A]{Am}; this describes the characteristic coefficients
of a linear combination $\sum \beta_i r_i$ of matrices in the
subalgebra generated by the characteristic coefficients of
products of the $r_i$.
\end{proof}

\begin{rem}
If $f(y_1, y_2, \dots)$ is $A$-quasi-linear in $y_1$ and $g(x_1,
x_2, \dots)$ is $(A;t;\bar q)$-quasi-alternating, then $$f(g(x_1,
x_2, \dots)y_1, y_2,\dots), \qquad f(y_1g(x_1, x_2, \dots),
y_2,\dots)$$ are $(A;t;\bar q)$-quasi-alternating.
\end{rem}

\begin{rem}[{\cite[Remark~G, p.~25]{BR}}]
Let $f(x_1, \dots, x_{t+1};y)$ be any $(A;t;\bar
q)$-quasi-alternating polynomial. Then the polynomial
\begin{equation}\label{alt1} \sum _{i=1}^{t+1} (-1)^i f(x_1,
\dots, x_{i-1}, x_{i+1}, \dots, x_{t+1}, x_i ; y)\end{equation} is
$(t+1)$-$A$-quasi-alternating.
\end{rem}

 \subsubsection{Zubrilin's theory applied to quasi-linear
polynomials}

Even without the trick of Remark~\ref{noZub}, we could resort to a
more polynomial-oriented approach developed by~Zubrilin, expounded
in \cite{BR}, which generalizes \cite[Theorem~J, p.~25]{BR}. Since
Zubrilin's theory as developed in \cite{BR} requires us to start
with multilinear polynomials and we must cope with quasi-linear
polynomials in this paper, we need to extend the theory to
quasi-linear polynomials. The idea is to take the characteristic
polynomial of the matrix $a^{\bar q}$ instead of the
characteristic polynomial of the matrix $a.$ Zubrilin's theory can
be considered to be the case $\bar q=1,$ and extends readily to
the general case. Unfortunately, the theory as developed in
\cite{BR} requires many computations, so here we only indicate
where the proofs are modified in this more general situation.

Recall \cite[Definition~2.40]{BR} that if $f(x_1,\dots,x_t;y)$ is
multilinear in the variables $x_i$, then
$(\delta_jf)(x_1,\dots,x_t;y;z)$ is the sum over all the possible
substitution of $zx_i$ for $x_i$ in $j$~out of the first $n$
places.
Explicitly,
let $f(x_1,\ldots, x_{n}, \vec y,  \vec t)$ be multilinear in the $x_i$
(and perhaps involving additional variables summarized as $\vec{y}$ and $\vec{t}$).   Take $0\le k\le n$, and expand
$$
f^*=f((z+1)x_1,\ldots,(z+1)x_n,\vec y, \vec t),
$$
where  $z$ is a new variable. Then we write
$\delta^{(x,n)}_{k,z}(f) := \delta^{(x,n)}_{k,z}(f)(x_1,\ldots,
x_{n},z)$ for the homogeneous component of $f^*$ of degree $k$
in the noncommutative variable $z$.

\begin{prop}[{\cite[Corollary~2.45]{BR}, \cite{zubrilin.1}}]\label{Zub}
Let $f(x_1, \dots, x_t, x_{t+1};y)$ be any $(A;t;\bar
q)$-quasi-alternating polynomial which is linear in $x_{t+1}$.
Also suppose the polynomial of \eqref{alt1} is an identity of $A$.
Then $A$ also satisfies the identity
\begin{equation}\label{alt2}
\sum_{j=0}^n (-1)^j \delta_{j,z}^{(n)}(f(x_1,\ldots,x_n,z^{n-j}x_{n+1}))\equiv 0 \quad
\text{modulo } \mathcal{CAP}_{n+1}.\end{equation}
\end{prop}

We will need to use Proposition~\ref{Zub} in the general case (Theorem \ref{SpechtNoeth}),
even though Remark~\ref{noZub} suffices for the field-theoretic case.

\section{Review of  quivers of representations}

Our main tool is the quiver of a representation, which we recall from   \cite{BRV2} and \cite{BRV3}/
(This differs from the customary definition of quiver, since it  is not Morita invariant but
takes into account the matrix size.)

\subsection{Full quivers and pseudo-quivers}

\begin{rem}\label{linop} Any representable algebra $A \subseteq \M[n](K)$ has
its Wedderburn block form described in detail in \cite{BRV1} and
\cite[Definition~3.10]{BRV2}, which is the keystone of
\cite{BRV2}. This Wedderburn block form induces an action of $A$
on $\M[n](K),$ by which we view each element $a\in A$ as a linear
operator $\ell_a$ on $V = K^n$ via left multiplication. (Likewise,
we also have a right action via right multiplication.) In the
sequel, we usually consider the algebra $A$ in this
context.\end{rem}

For further reference, we also bring in the slightly more general
notion of pseudo-quiver, to enable linear changes of basis in the
representation.

See \cite{BRV2} and \cite{BRV3} for details about full quivers and
pseudo-quivers. It is useful to formulate the definition purely
geometrically, without reference to the original algebra.

\begin{defn}
An \textbf{(abstract) full quiver} (as well as \textbf{(abstract) pseudo-quiver})
is a directed graph $\Gamma$, without double edges and without
cycles, having the following information attached to the vertices
and edges:

\begin{enumerate}
\item The vertices are ordered, say from $\bf 1$ to $\bf k$, and
an edge can only take a vertex to a vertex of higher order. There
also are identifications of vertices and of edges, called
\textbf{gluing}. Gluing of vertices is of one of the following
types: \begin{itemize}\item \textbf{Identical gluing}, which
identifies matrix entries in the corresponding blocks; \item
\textbf{Frobenius gluing}, which identifies matrix entries in one
block with their $q$-th power in another block, where $q$ is a
power of $p$;\item \textbf{Gluing up to infinitesimals} described
in \cite[Definition 2.3]{BRV3}; \item (in the case of
pseudo-quivers) Linear relations among the vertices,
cf.~Remark~\ref{myst} below.
\end{itemize}

 Each vertex is labelled with a roman numeral ($I$,
$\II$ etc.); glued vertices are labelled with the same
 roman numeral.

 The first vertex listed in
a glued component of vertices is also given a pair of subscripts
$(n_{\bf i},t_{\bf i})$: the \textbf{matrix degree} $n_{\bf i}$
and the \textbf{cardinality} of the corresponding field extension
of $F$.

 \item Off-diagonal
gluing (i.e., gluing among the edges) has several possible types,
including \textbf{Frobenius gluing} and \textbf{proportional
gluing} with an accompanying \textbf{scaling factor}. Absence of a
scaling factor indicates scaling factor 1; such gluing is called
\textbf{identical gluing} when there is no Frobenius twist
indicated.

Frobenius gluing of a block with itself and gluing up to
infinitesimals can be viewed as modifying the base ring, yielding
a commutative affine algebra over a field instead of a field.

Very briefly, the \textbf{quiver of a representation} is obtained
by taking the Wedderburn block form of the image, associating
vertices to the diagonal blocks and arrows to the blocks above the
diagonal. Gluing corresponds to identification of matrix
components in the algebra. The pseudo-quiver is obtained when we
make extra identifications of the vertices (which results in extra
gluing).

\end{enumerate}

\end{defn}

\begin{rem}\label{myst0}
(i) Any representation of an algebra $A$ into Wedderburn block form gives
rise to a full quiver. Starting with the $k$ vertices $v_1,
\dots, v_k$ corresponding to central idempotents of the blocks, we
proceed as in Remark~\ref{linop}. We identify the entries of the
respective blocks according to gluing (with identical gluing
identifying entries and with Frobenius gluing corresponding to the
Frobenius automorphism). For any two idempotents $e_{\bf i},e_{\bf
j}$ we choose a base of $e_{\bf i} A e_{\bf j}$ for the arrows
between ${\bf i}$ and ${\bf j}$. Then by definition, any two
consecutive vertices have only a single arrow joining them,
although now we must accept new gluing (corresponding to linear dependence) of vertices.

(ii) Conversely, any abstract full quiver gives rise to a $C$-subalgebra $A$ of $M_n(K)$ in Wedderburn block,
form, where we read off the diagonal blocks from the vertices
(together with the matrix size, base field, and gluing), and then
write down the off-diagonal parts from the arrows together with
the relations that are registered together with the quiver.

Note that this observation does not require that $C$ be a field.
In this way, we can define the algebra of a quiver over an
arbitrary integral domain.
\end{rem}

\begin{rem}\label{myst} Any full quiver of a representation of an algebra $A$ gives
 rise to a pseudo-quiver. Starting with the $k$ vertices $v_1,
 \dots, v_k$ corresponding to central idempotents of the blocks, we
 proceed as in Remark~\ref{linop}. We identify the entries of the
 respective blocks according to gluing (with identical gluing
 identifying entries and with Frobenius gluing corresponding to the
 Frobenius automorphism). For any two idempotents $e_{\bf i},e_{\bf
 j}$ we choose a base of $e_{\bf i} A e_{\bf j}$ for the arrows
 between ${\bf i}$ and ${\bf j}$. Then by definition, any two
 consecutive vertices have only a single arrow joining them,
 although now we must accept new gluing (corresponding to linear dependence) of vertices.
\end{rem}

\subsection{Degree vectors}

\begin{defn}\label{dv}
The {\bf length} of a path $\mathcal B$ in a pseudo-quiver is its
number of arrows, excluding loops, which equals its number of
vertices minus $1$. Thus, a typical path has vertices $r_1, \dots,
r_{\ell+1}$, where the vertex $r_j$ has matrix degree $n_j$. We
call $(n_1, \dots, n_\ell)$ the \textbf{degree vector} of the path
$\mathcal B$. We order the degree vectors according to the largest
$n_j$ which appears in the distinct glued components, counting
multiplicity. More precisely, for any degree vector we discard any
duplications (due to gluing), and then associate the number $d_k$
to the number of components of matrix degree $k$. We order the
degree vectors according to these sets of $d_k$, taken
lexicographically.

 For example, the degree vector $(3,1,3,3)$ with no gluing of vertices
is greater than $(3,2,2,3,2,3,3)$ with the fourth, sixth, and
seventh vertices glued since 3 appears three times unglued
 in the first degree vector but only twice in the second degree vector.
\end{defn}

\begin{rem}\label{dv2} We also define a secondary order on $\mathcal B$ with
respect to the grade defined above, because further gluing will
lower the number of elements in the grading monoid.
\end{rem}

\subsubsection{Degenerate gluing between branches}$ $

{\bf Degenerate gluing} is the situation in which each edge of one
branch is glued to the corresponding edge of another branch; then
the two branches produce the same values when we multiply out the
elements in the corresponding algebra. We can eliminate degenerate
gluing by passing to the pseudo-quiver, but it often is more convenient
for us to make use of \cite[Proposition~3.13]{BRV3}:

\begin{prop}\label{degglu}
Any representable, relatively free algebra has a representation
 whose full quiver has no degenerate gluing.
\end{prop}

Since we are working with T-ideals, which correspond to relatively
free algebras, this result enables us to bypass pseudo-quivers.

As shown in \cite[Lemma 2.8]{BRV3}, gluing of two vertices of a
pseudo-quiver can often be eliminated simply by joining the
vertices (since the arrows now are linear operators).
 We call this process \textbf{reducing} the pseudo-quiver,
 and we always assume in the sequel that our pseudo-quivers are reduced.

Conversely, given a quiver or pseudo-quiver $\Gamma$, one can take
an arbitrary commutative Noetherian algebra $K$ and build a
representable algebra $A$ into $\M[n](K)$ from $\Gamma$. One
theme of \cite{BRV3}
is how the
geometric properties of $\Gamma$ yield identities and
non-identities of~$A$.

Linear relations among the vertices of a pseudo-quiver can only
occur if all paths between these vertices have the same ``grade,''
as described in \cite[\S 2.7]{BRV3}. This becomes somewhat
intricate in characteristic $p$, in presence of Frobenius gluing,
so we review the idea here. Later on, we will need an inductive
procedure which applies to when the gluing is strengthened. In
principle, this is obvious, because any further gluing which does
not lower the degree vector must lower the power of $q$ used in
the Frobenius twist, and so this must terminate after a finite
number of steps. To state this formally requires some technical
details, which we review from \cite{BRV3}.

Take $A$ to be the Zariski closure of a representable relatively
free algebra, having full quiver $\Gamma$. We write $\MG_\infty$
for the multiplicative monoid $\set{1,q,q^2,\dots,\epsilon}$,
where $\epsilon a = \epsilon$ for every $a \in \MG_{m}$. (In other
words, $\epsilon$ is the ``zero'' element adjoined to the
multiplicative monoid $\langle q \rangle$.) When the base field
$F$ is infinite, the full quiver can be separated (replacing $\a$
by $\a \gamma$ where $\gamma^q \ne \gamma$), and the algebra can
be embedded in a graded algebra. In other words, each diagonal
block is $\MG_\infty$-graded, where the indeterminates~$\la_{\bf
i}$ are given degree 1; hence, $A$ is naturally $\bM$-graded.

When $F$ is a finite field of order $q$, $\MG_m$ denotes the
monoid obtained by adjoining a ``zero'' element $ \epsilon$ to the
subgroup $\langle q \rangle$ of the Euler group $U(\Z_{q^m-1})$,
namely $\MG_{m} = \set{1,q,q^2,\dots,q^{m-1}, \epsilon}$
 where $\epsilon a = \epsilon$ for every $a \in \MG_{m}$.
 Let
 $\bM$ be the semigroup $\MG /\!\!\sim$, where $\sim$ is the
equivalence relation obtained by matching the degrees of glued
variables: When two vertices have Frobenius gluing $\alpha \ra
\phi^i(\alpha)$, we identify $1$ with $q^i$ in their respective
components.

 By \cite[Lemma~2.7]{BRV3}, every $\MG_{m}$ is a quotient of $\MG_{\infty}$, and
more generally whenever $m \divides m'$, the natural group
projection $\mathbb Z_{q^{m'}-1} \to \mathbb Z_{q^m-1}$ extends to
a monoid homomorphism $\MG_{m'} \to \MG_{m}.$

The diagonal blocks $S_r$ of $A$ (under multiplication) can be
viewed as $\MG_{t_r}$-modules, where we define the product
$[q^i]a$ to be $a^{q^i}$, and $[\epsilon]a = 0$. $S_r$ itself is
not graded, and we need to pass to the larger algebra $B$ arising
from the sub-Peirce decomposition of~$A$, cf.~\cite[
Remark~2.33]{BRV3}.

Let $A_{r,r'}$ denote the $(r,r')$-sub-Peirce component of $A$.
Then $A_{r,r'}$ is naturally a left $F_r$-module and right
$F_{r'}$-module, so $A_{r,r'}$ is graded by the monoid $\MG_{\hat r}$ obtained from $F_{\hat r}$, the compositum of $F_r$ and
$F_{r'}$. (In other words, if $F_r$ has $q^t$ elements and
$F_{r'}$ has $q^{t'}$ elements, then $F_{\hat r}$ has $q^{\hat t}$
elements, where $\hat t = \operatorname{lcm}(t,t').$) Since the
free module is graded and the monoid $\MG_{\hat r}$ is invariant
under the Frobenius relations, we see that the Frobenius relations
preserve the grade under $\MG_{\hat r}$.

\subsection{Canonization Theorems}

\begin{defn}
A full quiver (resp.~ pseudo-quiver) is \textbf{basic} if it has a
unique initial vertex $r$ and unique terminal vertex $s$. A basic
full quiver (resp.~ pseudo-quiver) $\Gamma$ is \textbf{canonical}
if it has vertices $r'$ and $s'$ satisfying the following
properties:

\begin{itemize}
\item $\Gamma$ starts with a unique path $p_0$ from $r$ to $r'.$
Thus, every vertex not on $p_0$ succeeds~$r'$. \item $\Gamma$ ends
with a unique path $p_0'$ from $s'$ to $s.$ Thus, every vertex not
on $p_0'$ precedes~$s'$.
\item Any two paths from the vertex $r'$ to the vertex $s'$ have
the same grade.

\end{itemize}

 An \textbf{enhanced canonical full quiver (resp.~pseudo-quiver)} is a canonical
full quiver (resp.~pseudo-quiver) with uniform grade.
\end{defn}

We recall the following ``canonization theorems:''

\bigskip
\begin{itemize}
\item \cite[Theorem~6.12]{BRV2}. Any relatively free affine
PI-algebra $A$ has a representation for whose full quiver all
gluing is Frobenius proportional.

\item \cite[Theorem~3.5]{BRV3}. Any
basic full quiver (resp. pseudo-quiver) of a relatively free
algebra can be modified (via a change of base) to an enhanced
canonical full quiver (resp.~ pseudo-quiver).

\item \cite[Corollary~3.6]{BRV3}. Any relatively free algebra is a
subdirect product of algebras with representations whose full
quivers (resp.~pseudo-quivers) are enhanced canonical.

\item \cite[Theorem~3.12]{BRV3}. For any $C$-closed T-ideal of a
relatively free algebra $A$, the full quiver of $A'= A/\mathcal I$
is obtained by means of the following elementary operations on the
full quiver of $A$: Gluing, new linear dependences on the
vertices, and new relations on the base ring.
\end{itemize}

Thus, the PI-theory can be reduced to the case of
enhanced canonical full quivers. Accordingly, all of the full
quivers and pseudo-quivers that we consider in this paper will be
enhanced canonical.

\section{Evaluations of polynomials arising from algebras of full quivers}

We get to the crucial point of this paper, which is how to
evaluate a polynomial $f(x_1,x_2 ,\dots )$ on a representable
algebra $A$, in terms of the full quiver $\Gamma$ of $A$. This
question is quite difficult in general, but we note that the
quasi-linearization of $f$ (as defined above) is in the T-ideal
generated by $f$, so we may assume that $f$ is quasi-linear. Thus,
by Remark~\ref{sub1}, the evaluations of a quasi-linear polynomial
$f(x)$ are spanned by the evaluations obtained by pure
specializations of the indeterminates $x_i$ to~$S\cup J$.

The reader should already note that all of the proofs of this section are
algorithmic, involving only a finite number of steps. This observation will be needed below in the
reduction of the base ring from an integral domain  to a field, in the second
proof of  Theorem~\ref{SpechtNoeth}.

\begin{rem}\label{expans}
Any nonzero evaluation arises from a string of substitutions $x_i
\mapsto \overline{x_i}$ to elements corresponding to some path of
the full quiver $\Gamma$. (We are permitted to have substitutions
repeating in the same matrix block, i.e., the vertex repeats via a
loop.) The $\overline{x_i}$ connect two vertices, say of matrix
degree $n_i$ and $n'_{i}.$ Suppose $t$ is the nilpotence index of
the radical $J$. Then any string involving $t$ radical
substitutions is~0. If we replace $x_i$ by $h_{n_i}x_{i,1}\cdots
x_{i,t} h_{n_i}x_{i}$, then we still get the same evaluation when
$x_{i,1}, \cdots, x_{i,t}$ are specialized to the identity
matrices in $n_i \times n_i$ matrix blocks, which in particular
are semisimple substitutions. Note that the $h_{n_i}$ were
inserted in order to locate the semisimple substitutions of the
$x_{i,1}$ inside matrix blocks of size at least $n_i \times n_i$.

On the other hand, the number of radical substitutions must be at
most the nilpotence index of $A$, so at least one of these extra
substitutions must be semisimple, if we are still to have a
nonzero evaluation. By taking $h_{n_i}x_{i,1}^{\bar
q}x_{i,2}\cdots x_{i,t} h_{n_i}x_{i}$ ($\bar q$ as in
Remark~\ref{qwhat}), we force this radical substitution to come
from $\overline{x_{i,1}}\, ^{\bar q}.$

Thus, this process eventually will yield a polynomial having a
nonzero specialization corresponding to a path whose degree vector
involves a semisimple substitution at each matrix block.\end{rem}

\subsection{Characteristic coefficient-absorbing polynomials inside
T-ideals}\label{trab}

The main goal of \cite{BRV3} was to show that any relatively free
representable algebra $A$ has a T-ideal in common with a
Noetherian algebra $\tilde A$ which is a finite module over a
commutative affine algebra when $A$ is affine. This T-ideal yields
a non-identity of a branch of $A$. The method was to define some
action of characteristic coefficients (i.e., coefficients of the
characteristic polynomial) of elements of a Zariski-closed algebra
$A$, such that the values of the nonidentity obtained from its
full quiver are closed under multiplication by these
characteristic coefficients. This enabled us to preserve
Hamilton-Cayley type properties in the evaluations of diagonal
blocks. Here we use the same techniques, but need to refine them
in order to obtain the desired polynomial within a given T-ideal
obtained from an arbitrary polynomial (not necessarily arising
from a branch).

Because we are working with quasi-linear polynomials instead of
multilinear polynomials, we must utilize only $\bar
q$-characteristic coefficients instead of all characteristic
coefficients. Let us recall (For example, \cite[Lemma~5.1]{BRV3}:

\begin{lem}\label{barq} Suppose $A$ is a representable algebra, over a field of characteristic $p$. When, for some $m$,
 $\bar q = p^m
>n$ (which is greater than the nilpotence index of the Jacobson radical), then
the element $a^{\bar q}$ is semisimple for every $a \in A$.
\end{lem}

This was enough to reduce various problems in \cite{BRV3} and
\cite{BRV4} to the characteristic 0 case, effectively replacing
$\bar q$ by 1, but does not suffice here to reduce Specht's
problem to the characteristic 0 case. Nevertheless, it is still a
key element of the proof, when we keep track of $\bar q$.

 Again, our objective is to apply the celebrated theorem of
 Shirshov \cite[Chapter 2]{BR} to adjoin $\bar q$-characteristic coefficients to $A$ and obtain
 an algebra finite as a module over a commutative affine
 $F$-algebra. For technical reasons, we only succeed in adjoining
 $\bar q$-powers of characteristic coefficients, so we formulate the following definition, modified from \cite{BRV3}:

\begin{defn}\label{absorp} Given a quasi-linear polynomial $f(x;y)$ in indeterminates
labelled $x_i,y_i$, we say $f$ is {\bf $\bar q$-characteristic
coefficient-absorbing} with respect to a full quiver $\Gamma =
\Gamma(A)$ if the following properties hold:

\begin{enumerate}
\item $f$ specializes to 0 under any substitution in which at
least one of the $x_i$ is specialized to a radical element of $A$.
(In other words, the only nonzero values of $f$ are obtained when
all substitutions of the $x_i$ are semisimple.)

\item $f(\mathcal A(\Gamma))^+$
 absorbs multiplication by any $\bar q$-characteristic coefficient of any
 element in a simple (diagonal)
 matrix block of $\mathcal A(\Gamma)$.
 \end{enumerate}
 \end{defn}

There are two ways of obtaining intrinsically the coefficients of
the characteristic polynomial $$f_a = \la^n + \sum_{k=1}^{n-1}
(-1)^k \a _j (a) \la ^{n-k}$$ of a matrix $a$. Fixing $k$, we
write $\a$ for $\a_k.$ (For example, if $k=1$ then $\a (a) = \tra
(a)$.)

Recall from \cite[Definition~2.23]{BRV3} that we defined
\begin{equation}\label{trmat}
\trmat(a) = \sum _{i,j = 1}^n
e_{ij} a e_{ji},\end{equation} called the {\bf matrix definition}
of trace. We need a generalization.

\begin{defn} \label{trmat0} In any matrix ring $\M[n](W)$, we define
\begin{equation}
 \alphaj(a) : = \sum _{j=1}^n \sum e_{j,i_1} a e_{i_2,i_2}a \cdots a e_{i_{k}i_k} a
 e_{i_1,j},\end{equation}
the inner sum taken over all vectors $(i_1,\dots,i_k)$ of length
$k$.
\end{defn}

Of course, these characteristic coefficients $ \alphaj(a) $
commute iff $W$ is a commutative ring. This is a key issue, that
we will need to address.

\begin{lem}\label{L1}
For any $\M[n](F)$-quasi-linear polynomial $f(x_1, x_2, \dots)$
which is also $\M[n](F)$-quasi-homogeneous of degree $\bar q$ in
$x_1$, the polynomial
$$\hat f = f(c_{n^2}(y) x_1 c_{n^2}(z), x_2, \dots)$$ is $\bar q$-characteristic coefficient
absorbing in $x_1$.
\end{lem}
\begin{proof} The same proof as in \cite[Theorem~J,
Equation~1.19, page~27]{BR}, when the assertion is formulated as:
\begin{equation}\label{traceab00}
\alpha_k ^{\bar q}f(a_1, \dots, a_t, r_1, \dots, r_m) = \sum
f(T^{k_1}a_1, \dots, T^{k_t}a_t, r_1, \dots, r_m),\end{equation}
summed over all vectors $(k_1, \dots, k_t)$ with each $k_i \in \{
0, 1\}$ and $k_1 + \dots + k_t = k,$ where $\alpha_k $ is the
$k$-th characteristic coefficient of a linear transformation $T \co V
\to V,$ and $f$ is $(A;t;\bar q)$-quasi-alternating.
\end{proof}

\begin{rem}\label{CHid}
Notation as in \eq{traceab00}, the Cayley-Hamilton identity for
$n_i \times n_i$ matrices which are evaluations of $f$ is
$$0 = \sum_{k=0}^{n_i}
\alpha_k ^{\bar q}f(a_1, \dots, a_t, r_1, \dots, r_m) =
\sum_{k_1,\dots,k_t} f(T^{k_1}a_1, \dots, T^{k_t}a_t, r_1, \dots,
r_m),$$ which is thus an identity in the T-ideal generated by $f$.
\end{rem}

Note that this is the same argument as used by Zubrilin in the proof of Proposition~\ref{Zub}.

Iteration yields:

\begin{prop}\label{q2} For any polynomial $f(x_1, x_2,
\dots)$ quasi-linear in $x_1$ with respect to a matrix algebra
$\M[n](F)$, there is a polynomial $\hat f $ in the T-ideal
generated by $f$ which is $\bar q$-characteristic coefficient
absorbing.
\end{prop}

\begin{defn} Fixing $0 \le k <n,$ we denote the $k$-th $\bar q$-characteristic coefficient of
$a$, defined implicitly in \Lref{L1}, as $\aqpol(a)$.
(Strictly speaking, $k$ should be included in the notation, but
since $k$ is taken arbitrarily in our results, we do not bother to
specify it.)
\end{defn}

\begin{defn}\label{HCind}
We call the identity $\sum_{k_1,\dots,k_t} f(T^{k_1}x_1, \dots,
T^{k_t}x_t, r_1, \dots, r_m)$ obtained in \Rref{CHid}, the {\bf
Hamilton-Cayley identity induced by $f$}.
\end{defn}

The following result holds for arbitrary algebras of paths.

\begin{rem}\label{trac3} We need an action of matrix characteristic coefficients (computed
on the diagonal components of the given representation of $A$) on
the T-ideal of~$f$. To do this, one computes the characteristic
coefficient as in Definition~\ref{trmat0}, and applies this on
each (glued) matrix component, i.e., the Peirce component
corresponding to vertices on each side of an arrow. More
precisely, suppose we have two Peirce components, whose
idempotents are $e_r = \sum_k e_{r,k}$ and $e_s = \sum_\ell
e_{s,\ell}$. For any arrow $\alpha$ from (non-glued) vertices
$r_i$ to $s_\ell$, we consider the matrix $(a_{uv})$ corresponding
to $\alpha$, and take characteristic coefficients on the
$r_k$-diagonal component on the left, and the $s_\ell$-diagonal
component on the right. In other words, if the vertex
corresponding to $r$ has matrix degree $n_i$, taking an $n_i
\times n_i$ matrix $w$, we define ${\aqpol} _u(w)$ as in the
action of \Lref{L1} and then the left action
\begin{equation}\label{mtr1}
a_{u,v} \mapsto {\aqpol}_u (w) a_{u,v}.
\end{equation}
Likewise, for an $n_j \times n_j$ matrix $w$ we define the right
action
\begin{equation}\label{mtr2} a_{u,v} \mapsto a_{u,v} {\aqpol} _v (w).
\end{equation}
(However, we only need the action when the vertex is non-empty; we
forego the action for empty vertices.)
\end{rem}

 We can proceed further whenever these two
$\bar q$-characteristic coefficient actions coincide on the
T-ideal of $f$.

\begin{lem} \label{onecomp0}
 $\alphaj(a)^{\bar q} = \aqpol(a)$ in
$\M[n](C)$ for $C$ commutative.

Thus, left multiplication by $\alphaj(a)$ acts on the set of
evaluations of any $n_i^2$-alternating polynomial $f(x;y)$ on an
$n_i \times n_i$ matrix component. \end{lem}
\begin{proof} Follows at once from Equation ~\eqref{traceab00}.
\end{proof}

\subsection{Identification of matrix actions for unmixed substitutions}
Our main objective is to introduce $\bar q$-characteristic
coefficient-absorbing polynomials corresponding to all
 canonical full quivers, in order to
identify these two notions of $\bar q$-characteristic
coefficients, working with matrix substitutions inside a given
polynomial. We have the natural bimodule action of $\bar
q$-characteristic coefficients on $A$ given in terms of the full
quiver, which we can identify with $\aqpol(a) ,$ defined in
Equation ~\eqref{traceab00}, whenever the matrix characteristic
coefficients commute. The theory subdivides into two cases:
\begin{itemize} \item The
substitution of an indeterminate is to sums of elements in the
same glued Wedderburn component. \item The substitution of an
indeterminate to sums of elements in the different glued
Wedderburn components. \end{itemize} The techniques are different.
We start with the first sort of situation, which we call
``unmixed,'' which can be treated via the argument of
Example~\ref{onecomp0}. For the second sort of situation, which we
call ``mixed,'' \Lref{sub2} is applicable, but requires an
intricate ``hiking procedure'' (defined presently) on the
quasi-linearization of a polynomial.
\begin{rem}\label{onecomp}
The argument of \Lref{onecomp0} holds for a single diagonal
matrix component over a commutative ring.
\end{rem}

Recall that the T-\textbf{space} of a polynomial $f$ on an
$F$-algebra $A$ is defined as the $F$-subspace of $A$ spanned by
the evaluations of $f$ on $A$. Ironically, the sophisticated
hiking procedure fails to handle the unmixed case since it relies
on a T-space argument which thus must fail in view of Shchigolev's
counterexample for ACC for T-spaces \cite{Sh}. So the theory
actually requires separate treatment of the ``degenerate'' unmixed
case. The proof of Remark~\ref{onecomp} involves the full force of
T-ideals rather than T-spaces.

\subsection{Hiking}

We arrive at the main new idea of this paper. Let $A$ be a Zariski
closed algebra, and let $f$ be a quasi-linear non-identity. The
goal is to replace $f$ by a better structured non-identity in its
T-ideal, for which the $\bar q$-characteristic coefficients of the
matrix blocks defined in components of the full quiver commute
with each other and also with radical substitutions of arrows
connecting glued vertices. This enables us to compute these $\bar
q$-characteristic coefficients in terms of polynomial evaluations.
We must cope with the possibility that our semisimple substitution

has been sent to the `wrong' component, either because its matrix
degree is too large or the base field is of the wrong size.

We write $[a,b]$ for the additive commutator $ab-ba,$ and
$[a,b]_q$ for the \textbf{Frobenius commutator} $ab -b^q a$.
\begin{lem}\label{hike00} If $f(x_1, \dots, x_n)$ is any
 polynomial quasi-linear in $x_i$, then \begin{equation}\label{comm1} f(a_1,
\dots, [a,a_i],\dots a_n) = f(a_1, \dots, aa_i,\dots, a_n)- f(a_1,
\dots, a_i a,\dots, a_n).\end{equation} and, more generally,
\begin{equation}\label{comm11} f(a_1, \dots, [a,a_{i_1}\cdots
a_{i_k}],\dots a_n) = \sum _{j=1}^k f(a_1, \dots, a_{i_1}\cdots
[a, a_{i_j}]\cdots a_{i_k},\dots a_n).\end{equation} for all
substitutions in $A$.
\end{lem}
\begin{proof}
By quasi-linearity, we may assume that $f$ is a monomial, in which
case we see that all of the intermediate terms cancel.
\end{proof}
We recall the hiking procedure of \cite[Lemma~5.8]{BRV3}, but the
hiking procedure here is quite subtle, and requires four different
stages.

\subsubsection{Stage 1 hiking}

\begin{lem}\label{prehike} Suppose a quasi-linear nonidentity $f$ of a Zariski closed algebra $A$ has a
nonzero value for some semisimple substitution of some $x_i$ in
$A$, corresponding to an arrow in the full quiver whose initial
vertex is labelled by $(n_i, t_i)$ and whose terminal vertex is
labelled by $(n_i', t_i')$. Replacing $x_i$ by $[x_i, h_{n_i}]$
(where the $h_{n_i}$ involve new indeterminates) yields a
quasi-linear polynomial
\begin{equation}\label{hik01} f(\dots, [x_i, h_{n_i}], \dots)\end{equation} in which any
substitution of $x_i$ into this diagonal block yields 0.\end{lem}
\begin{proof} The evaluations of $h_{n_i}$ in the semisimple part
are central; hence, any nonzero value in $ f(\dots, [x_i,
h_{n_i}],\dots)$ forces us into a radical substitution.\end{proof}

\begin{lem}\label{hike} For $f$ as in \Lref{prehike},
\begin{equation}\label{hik02} \nabla_i f := f(\dots, [x_i, h_{\max\{n_i, n_{i'}\}}],
\dots)\end{equation} also does not vanish on $A$. In the case of
Frobenius gluing $x \mapsto x^{q^\ell}$, we need to take instead
the substitution
$$x_i \to f (\dots, [x_i, h_{\max\{n_i, n_{i'}\}}]_{q^\ell}, \dots).$$
\end{lem}
\begin{proof} There are
substitutions in the appropriate diagonal block (the one whose
degree is $\max\{n_i, n_{i'}\}$) for which
 $h_{\max\{n_i, n_{i'}\}}$ is a nonzero scalar, and we specialize $x_i$ to an
element which passes from one block to the other. \end{proof}

\begin{cor}

 Any nonidentity can be hiked via successive stage 1 hiking to ensure semisimple substitutions
 in each matrix component.
\end{cor}
\begin{proof} Once we have a nonzero substitution of~$f$ with external radical
substitutions in all the hiked positions, we may continue to hike
as much as we want without affecting the fact that we have a
nonzero substitution, i.e., that $f$ is a nonidentity.\end{proof}
\begin{exmpl}
Stage 1 hiking is illustrated via the full quiver given for the
Grassmann algebra on two generators: \begin{equation}\label{G2+}
\xymatrix@C=40pt@R=32pt{ %
I \ar@/^0pt/[r]^{\alpha} \ar@/_0pt/[d]^{\beta}%
& %
I \ar@/^0pt/[d]^{-\beta} %
\\
I \ar@/^0pt/[r]^{\alpha}%
& %
I %
}
\end{equation}
 Clearly the
critical nonidentity for each branch is $[x_1, x_2],$ and we get
the Grassmann identity $[[x_1, x_2],x_3]$ by taking $f= x_1$ and
hiking.

\end{exmpl}

We call this procedure (application of \eqref{hik02})
 \textbf{stage 1} hiking, since we also need other forms
of hiking which we call \textbf{stage 2}, \textbf{stage 3} and
\textbf{stage 4} hiking.

\begin{rem}

Stage 1 hiking absorbs all internal radical substitutions, cf.~
Definition~\ref{intern}, because of the use of the central
polynomial $h_{\max\{n_i, n_{i'}\}},$ so when working with fully
hiked polynomials, we need consider only the Peirce decomposition
(and not the more complicated sub-Peirce decomposition, see
\cite{BRV1}.) In this manner, stage 1 hiking leads us to external
radical substitutions for $x_i$, say from a block of degree $n_i$
to a block of degree $n_{i+1}.$
\end{rem}

Explicitly, after stage 1 hiking, we have obtained expressions of
the form
\begin{equation}\label{hik2} g_i(x,y,z) =
z_{i,1}[ h_{\max\{n_i, n_{i'}\}}(x_{i,1},
x_{i,2},\dots),y_i]z_{i,2}.\end{equation} To simplify notation, we
assume $n_{i+1} = n_i'.$ Now we define
\begin{equation}\label{fullv1} \tilde f
 = f(h_{n_1}, g_1, h_{n_2}, g_2, \cdots, g_{\ell} ,h_{n_{\ell}+1}),\end{equation} where different
indeterminates are used in each polynomial, in which we get the
term \begin{equation}\label{fullv2} h_{n_1} g_1 h_{n_2} g_2 \cdots
g_{\ell}
 h_{n_{\ell}+1}.\end{equation} 

Since the radical of a Zariski closed algebra $A$ is nilpotent, we
can perform stage 1 hiking on $f$ only at a finite number of
different positions (bounded by the nilpotence index of $A$)
before getting an identity. Stopping before the last such hike
gives us a nonidentity which would become an identity after any
further hike of stage 1.

Unfortunately, our polynomial~$f$ has several different monomials,
and when we hike with respect to one of these monomials, some
other monomial will give us some permutation of \eqref{fullv2}, in
which the substitutions might go into the ``wrong component.'' The
difficulty that we will encounter is that any matrix component can
be embedded naturally into a larger matrix component, so a given
matrix substitution could be viewed as being in this larger
component, thereby ruining our attempts to compute with
polynomials on each individual matrix component. Indeed, even a
radical substitution could be replaced by a semisimple
substitution in a larger component.

\subsubsection{Stage 2 hiking}

In view of the previous paragraph, we also need a second stage of
hiking, to
take care of substitutions into the ``wrong'' component.

 Given a nonzero specialization of a monomial of $f$ under the substitutions
$x_i \mapsto \overline{x_i}$, $i \ge 1,$ where $\overline{x_i} \in
\M[n_i](K)$, consider the specialization of another (permuted)
monomial of $f$ under the substitutions $x_i \mapsto
\overline{x_i}'$, $i \ge 1,$ where $\overline{x_i}' \in
\M[n_j](K)$ (and perhaps $j \ne i$).

\begin{exmpl}
Consider the algebra
$$\set{\left(\begin{matrix} \a & * & *& *
\\ 0 & \beta & * & * \\ 0 & 0 & * & *
\\
 0 & 0 & * & *
 \end{matrix}\right) : \a, \beta \in F},$$
where $*$ denotes an arbitrary element in $K$. The corresponding full quiver
$I_{(1,1)} \to \II_{(1,1)}\to \III_2$ would normally give
us the polynomial $$ z_{1,1}[ x_{1,1}
 ,y_1]z_{1,2}z_{2,1}[ h_2(x_{2,1}, \dots)
 ,y_2] z_{3,1},$$
 which could be condensed to $ [ x_{1,1}
 ,y_1 ]z [ h_2(x_{2,1}, \dots)
 ,y_2] $ since various
indeterminates can be specialized to 1. But if $f(x_1,x_2,x_3,
\dots)$ has both monomials $x_1x_2x_3\cdots $ and $x_3x_2x_1\dots
$ then hiking in the second monomial yields the permuted term
$$[ h_2(x_{2,1}, \dots)
 ,y_2]z[ x_{1,1}
 ,y_1 ]$$
which permits a nonzero evaluation with all substitutions in the
lower $2\times 2$ matrix component, and we cannot get a proper
hold on the substitutions.
\end{exmpl}

We need to hike $f$ further, to guarantee that the specialization
of some unintended monomial of~$f$ in \eqref{fullv1} does not land
in a subsequent matrix component $\M[n_j](K)$ for $n_j>n_i$; our
next stage of hiking eliminates all such specializations.

\begin{lem}\label{hike2} Let $H$ denote the
 central polynomial $h_{n_{j}}^{\bar q}$
of $M_{n_j}(K),$ and take
$$z_{i,1}[ h_{n_i}(x_{i,1}, x_{i,2},\dots),y_i]z_{i,2}g_{i+1} \cdots g_{j-1} H^{q_1}
 - z_{i,1}H^{q_2}[ h_{n_i}(x_{i,1}, x_{i,2},\dots), y_i]z_{i,2}g_{i+1} \cdots g_{j-1},$$
(for each pair $(q_1,q_2)$ that occurs in Frobenius twists in the
branch; in characteristic 0 we would just take $q_1=q_2= 1$.) The
product of these terms, taken over all the $q_1$ and~$q_2$,
becomes nonzero iff the substitution is to the $i$ component (of
matrix size $n_i$.)
\end{lem} \begin{proof} Specializing this expression into the $j$-component
(of size $n_j$) would yield two equal terms which cancel, since
$H$ takes on scalar values, and thus yield 0. But specializing
into the $i$ component (of size $n_i$) would yield one term
nonzero and the other~0 since $H$ is an identity on $n_i \times
n_i$ matrices, so their difference would be nonzero.
\end{proof}

In this way, we eliminate the ``wrong'' specializations in other
monomials of $f$ while preserving the ``correct'' ones. The
modification of $f$ according to this specialization is called
\textbf{stage 2 hiking of the polynomial~$f$}.

\subsubsection{Stage 3 hiking}
So far we have guaranteed the specializations to be in the matrix
components of the correct size, but we need to fine tune still
further because the centers of the components may be of different
sizes. Next, we will want to reduce to the case where for any two
branches, the base field for each vertex has the same order.
Suppose $\mathcal B'$ is another branch with the same degree
vector. If the corresponding base fields for the $i$-th vertex of
$\mathcal B$ and $\mathcal B'$ are $n_i$ and $n_i'$ respectively,
we take $t_i = q^{n'_i}$ and replace $x_i$ by
$(h_{n_i}^{t_i}-h_{n_i})x_i.$ The modification of $f$ according to
this specialization is called \textbf{stage~3 hiking of the
polynomial~$f$}.

\subsubsection{Stage 4 hiking}

Some of the radical substitutions are \textbf{internal} in the
sense that they occur in a diagonal block (after ``gluing up to
infinitesimals''). Hiking absorbs all internal radical
substitutions, because of the use of the central polynomial~
$h_{n_i},$ so when working with fully hiked polynomials, we need
consider only the Peirce decomposition (and not the more
complicated sub-Peirce decomposition; cf.~\cite{BRV1}.)

\begin{lem}\label{modhike1} There is a substitution to hike $f$ further such that
\begin{equation}\label{hikemore0}\tilde c_{n_i^2}(y) x_i
c_{{n'_i}^2(y)}\tilde c_{n_i^2}(\a _k y) x_i c_{{n'_i}^2(\a _k y)}
- \tilde c_{n_i^2}(\a _k y) x_i c_{{n'_i}^2(\a _k y)} \tilde
c_{n_i^2}(y) x_i c_{{n'_i}^2(y)}\end{equation} vanishes under any
specialization to the $n_i, n'_i$ blocks. \end{lem}\begin{proof}
By Proposition~\ref{q2} there is a Capelli polynomial $\tilde
c_{n_i^2}$ and $p$-power $\bar q$ such that
\begin{equation}\label{CC=CC}
\tilde c_{n_i^2}( \a_k y) x_i c_{{n'_i}^2}(y) = \a_k ^{\bar
q}(y_1)c_{n_i^2}(y) x_i c_{{n'_i}^2}(y)\end{equation} on the
diagonal blocks.

Since $\bar q$-characteristic coefficients commute on any diagonal
block, we see from this that
\begin{equation}\label{hikemore}\tilde c_{n_i^2}(y) x_i
c_{{n'_i}^2(y)}\tilde c_{n_i^2}(\a _k y) x_i c_{{n'_i}^2(\a _k y)}
- \tilde c_{n_i^2}(\a _k y) x_i c_{{n'_i}^2(\a _k y)} \tilde
c_{n_i^2}(y) x_i c_{{n'_i}^2(y)}\end{equation} vanishes
identically on any diagonal block, where $z = \a_k y$.

One concludes from this that substituting~\eqref{hikemore} for
$x_i$ would hike our polynomial one step further. But there are
only finitely many ways of performing this hiking procedure. Thus,
after a finite number of hikes, we arrive at a polynomial in which
we have complete control of the substitutions and the $\bar
q$-characteristic coefficients commute.
\end{proof}

\subsection{Admissible polynomials}$ $

Although hiking is the key tool in this analysis, one must note
that the only time that the hiking procedure makes the two actions
of Remark~\ref{trac3} coincide is when it is applied to the
components of maximal matrix degree. In other words, the
polynomial $f$ is required to have a nonzero evaluation on a
vector of maximal matrix degree. Thus we must consider the
following sort of polynomial.

\begin{defn} A polynomial $f$ is $A$-\textbf{admissible} on a Zariski-closed algebra $A$ if
$f$ takes on some nonzero evaluation on a vector of maximal matrix
degree. We denote such a vector as $v_\mathcal B $, where
$\mathcal B $ is the branch of the full quiver which gives rise to
$v_\mathcal B $, and call $v_\mathcal B $ the \textbf{matrix
vector of} $f$.
\end{defn}

If all of the substitutions of an indeterminate are to glued edges
of the full quiver, connecting vertices corresponding to the same
matrix degree $n_1$, then the notion of admissibility is
irrelevant, since one could just replace $f$ by $c_{n_1^2}f$ and
obtain the desired action on the matrix component from
equation~\eqref{traceab00}.

But there is a subtle difficulty. Without gluing, we could proceed
directly via Remark~\ref{trac3}. But gluing between branches leads
to complications in applying Remark~\ref{trac3}. For example, one
could have two strings $I \to \II \to \III$ and $I \to \III \to \II$,
whereby the substitutions in $f$ go to incompatible components.
The following definition, inspired by Drensky~\cite{D}, enables us
to bypass this difficulty in the proof of the next theorem.

\begin{defn}\label{sym} Given matrices $a_1, \dots, a_k,$ the \textbf{symmetrized} $(t;j)$ characteristic coefficient is the $j$-elementary symmetric function applied to the $t$-characteristic
coefficients of $a_1, \dots, a_k.$ \end{defn} For example, taking
$t=1$, the symmetrized $(1,j)$-characteristic coefficients $\a_t$
are
$$\sum_{j=1}^k \tr(a_j),\quad \sum _{j_1 >j_2} \tr(a_{j_1})\tr(a_{j_2}),\quad \dots , \quad \prod _{j=1}^k \tr(a_{j}).$$

\begin{thm}\label{traceq2}
Suppose $f $ is an $A$-admissible non-identity of
a representable, relatively free algebra~ $A$. Then the T-ideal
$\CI$ generated by $f$ contains a $\bar q$-characteristic
coefficient-absorbing $A$-admissible non-identity $\tilde f$.

 Furthermore, the T-ideal $\CI_\mathcal B$
 of all fully hiked $A$-admissible
polynomial obtained from the degree vector $v_\mathcal B$ is
comprised of evaluations of $\bar q$-characteristic
coefficient-absorbing polynomials, comprised of sums of
evaluations on pure specializations in $\mathcal B$.
\end{thm}
\begin{proof}
Let $\Gamma$ be the full quiver of $A$. We follow the proof of
\cite[Theorem~5.8]{BRV3}, where we follow the process of building
the polynomial $\Phi$ of \cite[Definition~4.11]{BRV3}, but we must
be more careful. Whereas in \cite[Theorem~5.12]{BRV3} we were
looking for any trace-absorbing non-identity and thus could hike
the polynomial of an arbitrary maximal path of a pseudo-quiver,
now we need to work within the polynomial $f$ belonging to a given
T-ideal and thus work simultaneously with all maximal paths of the
pseudo-quiver of the corresponding relatively free algebra.

Our polynomial $f$ need not be multilinear, although we can make
it $A$-quasi-linear with respect to $A$. Since the other
polynomials we insert are multilinear, $\tilde f$ remains
$A$-quasi-linear. But $f$ need not even be A-quasi-homogeneous, so
multiplying some variable by a $\bar q$-characteristic coefficient
might throw $f$ out of the T-ideal $\CI$.

 If the base field $F$ were
infinite, we could apply \Lref{quasi} to replace $f$ by an
A-quasi-homogeneous polynomial, but in general this cannot be done
so easily. Accordingly, we adopt the
following strategy in characteristic $p$:

We say that two strings in $A: = A(\Gamma)$ are
\textbf{compatible} if their degree vectors are the same. (This
means that the the matrix sizes of their semisimple substitutions
match.) Our overall goal is to modify $f$ so that all non-zero
substitutions in $f$ are compatible, and also to match the
Frobenius twists, thereby enabling us to define characteristic
coefficient-absorption.

\begin{itemize}
\item We work with symmetrized $\bar q$-characteristic coefficients instead of traces. (The reason is given before Definition~\ref{sym}; we take $\bar q$-powers to make sure that we are working with semisimple matrices.)

\item We pass to $\bar q$ powers of transformations for a
suitable power $\bar q$ of $p$.

\item We pick a particular branch $\mathcal B$ of $A$ on which we
can place the substitutions of a nonzero evaluation of $f$, and
modify $\tilde f$ so that all branches that do not have the same
degree vector $v_\mathcal B $ as that of $\mathcal B$,
cf.~Definition~\ref{dv}, become incompatible.

\item We make branches incompatible to $\mathcal B$ when they do not have the same
size base fields as $\mathcal B$ for the corresponding vertices.

\item We eliminate all Frobenius twists which do match those of
$\mathcal B$.
\end{itemize}

The hiking procedure only involves a finite number of polynomials, and
is all implemented in the following technical lemma:
\begin{lem}[Compatibility Lemma]\label{complem}
For any $A$-admissible
 non-identity $f$ of a representable Zariski-closed algebra
$A$, the T-ideal $\CI$ generated by the polynomial $f$ contains a
symmetrized $\bar q$-characteristic coefficient-absorbing
polynomial $\bar f$, not an identity of $A$, in which all
substitutions providing nonzero evaluations of $f$ are compatible.
\end{lem}
\begin{proof}
If all the vertices of the full quiver $\Gamma$ of $A$ are glued,
we are done by Remark~\ref{onecomp}. Thus, we assume that
$\Gamma$ has non-glued vertices, and thus has strings with degree
vectors of length $>1$.

Since applying the hiking procedure of Remark~\ref{modhike1} does
not change the hypotheses, we may assume that $\tilde f$ is fully
hiked. We choose $\bar q$ according to
\Lref{barq}.

 We consider all substitutions $x_i \mapsto \overline{x_i } $ to
semisimple and radical elements. Of all substitutions which do not
annihilate $f$, some monomial of $f$ then specializes to a nonzero
evaluation, i.e., a path in the full quiver, and we choose such a
substitution whose path $\mathcal P$ has maximal degree vector,
and, after modifying $f$ along the lines of Remark~\ref{expans},
we may assume that the largest component $n_i$ of the degree
vector involves some semisimple substitution which, with slight
abuse of notation, we denote as $\overline{x_i}.$ We want to make
any other substitution $x_i \mapsto \overline{x_i' } $ compatible
with our given substitution.

Suppose first that $ \overline{x_i }$ is a substitution connecting
vertices of matrix degree $n_{i,1}$ and $n_{i,2}$. In particular,
since any two consecutive arrows have a common vertex, $n_{i,1} =
\overline{x_i }$ must be an external radical substitution; if $
\overline{x_i }$ is a semisimple substitution or an internal
radical substitution, then $n_{i,1} = n_{i,2}$.

Take any substitution $ \overline{x_i '}$, connecting vertices of
degree $n_{i,1}'$ and $n_{i,2}'$. Replacing $x_i$ by $h_{n_i} x_i
h_{n_i}$ would annihilate any substitution to a component of
smaller matrix degree, so we may assume that $n_{i,1}' \ge n_{i,1}
$ and $n_{i,2}' \ge n_{i,2} .$

Take $i$ such that $n_i$ is maximal. Since the path $\mathcal P$
is assumed to have maximal degree vector, we must have $n_{i,1}' =
n_{i} = n_{i,2}'.$ Furthermore, replacing $x_i$ by $x_i^{\bar q}$
forces the semisimple substitution $ \overline{x_i '}^{\bar q}.$

We will force all our new substitutions $ \overline{x_j '}$ to be
compatible with the original ones along $\mathcal P$, by working
around this vertex.

Inductively (working backwards), we assume that $n_{k}' = n_{k } $
for all $j < k \le i.$ We already know that $n_{j }' \ge n_{j }$,
but want to force $n_{j }' = n_{j }$. It is enough to check this
for any external radical substitution $ \overline{x_j }$, since
these fix the degree vectors.

If $\overline{x_j}'$ also is an external radical substitution,
then $n_{j }' = n_{j }$ by maximality of the degree vector, so we
are done unless $\overline{x_j}'$ is in the diagonal matrix block
of degree $n_k \times n_k$; in other words, $n_{j }' = n_k$
whereas $n_j < n_k.$

Now we apply stage 2 hiking (as described in \Lref{hike2}),
which preserves the compatible substitutions and annihilates the
incompatible substitutions.

Thus, we only need concern ourselves with substitutions
$\overline{x_j }$ into the same matrix component along the
diagonal.

 Since $f$ is presumed to be fully hiked, for any internal
radical substitution $ \overline{x_j }$, we may assume that $
\overline{x_j}'$ also is an internal radical substitution;
otherwise further hiking will not affect an evaluation along
$\mathcal P$ but will make the other evaluation 0.

As noted above, when $ \overline{x_j }$ is a semisimple
substitution, taking $x_j ^{\bar q}$ forces the semisimple
substitution $ \overline{x_j '}^{\bar q}.$

In conclusion, by modifying $f$ we have forced all the
substitutions $ \overline{x '}$ to be compatible with the original
substitutions $ \overline{x }$.

Next, we want to reduce to the case where for any two branches,
the base field for each vertex has the same order. Suppose
$\mathcal B'$ is another branch with the same degree vector, but
with a base field of different order. Stage 3 hiking will zero out
all substitutions of $x_i$ in $\mathcal B'$, and thus make
$\mathcal B'$ incompatible.

 Since by definition any further hike of $\tilde f$ yields an identity,
Remark~\ref{modhike1} dictates
 that polynomial $\bar q$-characteristic coefficients defined in terms of $f$ via Equation \eqref{traceab00}
must commute. There is a subtlety involved when the sub-Peirce
component involves consecutively glued vertices, since then we
need the radical substitution $b$ of an arrow connecting glued
vertices to commute with $\aqpol (a)$. But this is because taking
any commutator hikes the polynomial further and thus is 0.

Now, as before, Remark~\ref{modhike1} dictates that polynomial ,
$\bar q$-characteristic coefficients defined in terms of $f$ via
Equation \eqref{traceab00} must commute, and thus the symmetrized
$\bar q$-characteristic coefficients commute, and applying
\Lref{hike00} shows that they commute with any radical
substitutions of arrows connecting glued vertices. Thus, the
polynomial action on $\CI$ coincides with the well-defined matrix
action as described in Remark~\ref{trac3}(3).

We still might encounter two different branches with the same
degree vector but with ``crossover gluing;'' i.e., $I \to \II \to
\III$ together with $\III \to \II \to I.$ Applying the same
polynomial $f$ to these two different branches might produce
different results. To sidestep this difficulty, assuming there are
$k$ such glued branches, we consider all possible matrices $a_1,
\dots, a_k$ appearing in the corresponding position of these $k$
branches, and take the elementary symmetric functions $\sigma_1,
\dots, \sigma_k$ on their $\bar q$-characteristic coefficients; in
other words, we take the symmetrized $\bar q$-characteristic
coefficients, as defined above. Now the polynomial action clearly
is the same on the $\sigma_j$, and thus on all symmetric functions
on the characteristic coefficients of $a_1, \dots, a_k$.

 Since the nonidentities all contain an $n_i^2$-alternating
polynomial at the component of matrix degree $n_i$ for each $i$,
and we apply the $\bar q$-characteristic coefficient action
simultaneously to each of these polynomials, their T-spaces
 are closed under multiplication by $\bar q$-characteristic coefficients of the
 simple components of semisimple substitutions, so we have a $\bar q$-characteristic coefficient-absorbing polynomial.
 \end{proof}

Note that to apply the lemma we have to take Frobenius gluing into
account. When substituting into blocks with Frobenius gluing, we
get characteristic coefficents with different Frobenius twists,
and then we do the symmetrization.

 The lemma
yields the proof of the first assertion \Tref{traceq2}, since we
have obtained the desired $\bar q$-characteristic
coefficient-absorbing $A$-admissible non-identity.

For the last assertion of the theorem, we note that the $\bar
q$-characteristic coefficient-absorbing properties of the lemma
were proved by means only of the properties of the degree vector
$v_\mathcal B$ and not on the specific polynomial $\tilde f$, and
$\tilde f$ vanishes for all pure substitutions to branches other
than $\mathcal B$. But $x \tilde f$ also is $\bar
q$-characteristic coefficient-absorbing with respect to the same
degree vector, and likewise the $\bar q$-characteristic
coefficient-absorbing properties pass to sums and homomorphic
images.
\end{proof}

We have completed the first step in our program:

\begin{thm}[$\bar q$-Characteristic Value Adjunction Theorem]\label{tradj}
Let $\hat A$ be the algebra obtained by adjoining to $A$ the
matrix symmetrized $\bar q$-characteristic coefficients of
products of the sub-Peirce components of the generic generators of
$A$ (of length up to the bound of Shirshov's Theorem~
\cite[Chapter 2]{BR}), and let $\hat C$ be the algebra obtained by
adjoining to $F$ these symmetrized $\bar q$-characteristic
coefficients. For any nonidentity $f$ of a representable
relatively free affine algebra $A$, the T-ideal $\CI$ generated by
the polynomial~$f$ contains a nonzero T-ideal which is also an
ideal of the algebra $\hat A$.

Also, $\hat A$ is a finite module over $\hat C$, and in particular
is Noetherian.
\end{thm}
\begin{proof} We follow the proof of \cite[Theorem~5.16]{BRV3}.
First we apply the partial linearization procedure to make $f$
$A$-quasi-linear. In view of \cite[Theorem~7.20 and
Corollary~7.21]{BRV1}, we may assume that the generators of $A$
are generic elements, say $X_1, \dots, X_t$. We adjoin the Peirce
components of these generic elements, noting that because the
polynomial obtained by Theorem~\ref{traceq2} is fully hiked, all
substitutions involving these Peirce components involve a product
of a maximal number of radical elements and thus still are in its
T-ideal $\CI$.

Let $\hat C'$ be the commutative algebra generated by all the
characteristic coefficients in the statement of the theorem. Then
$\hat C'$ is a finite module over $\hat C$, implying $\hat A$ is
finite over $\hat C$, in view of Shirshov's Theorem. \end{proof}
 Note that in
the affine case we can work with finitely many entries, and thus
$\hat C'$ is a finite module; this simple argument fails in the
non-affine case, which explains why this theory only applies to
affine algebras.

 We want to adjoin the matrix $\bar
q$-characteristic coefficients by means of evaluations of the
hiked polynomial $\tilde f.$ This can be done for the largest size
Peirce components of these new components (applied {\it
simultaneously} to each generator and each Peirce component) to
obtain an algebra $\hat A$, and we obtain a finite submodule via
the following modification of Shirshov's theorem:

\begin{defn}
Given a module $M$ over a $C$-algebra $A$ and a commutative
subalgebra $C_1\subset A$, we say that an element $a\in A$ is
\textbf{integral} over $C_1$ with respect to $M$ if there is some
monic polynomial $f \in C_1[\la]$ such that $f(c)M = 0$. $C_1$ is
\textbf{integral} with respect to $M$
if each element of $C_1$ is integral with respect to $M$. 
\end{defn}

\begin{thm}\label{Shir}  Suppose  $A = C\{ a_1, \dots, a_\ell\}$ and $M$ is an $A$-module, and
$A$ contains a commutative (not necessarily central) subalgebra
$C_1$ such that each word in the generators of length at most the
PI-degree is integral over $C_1$ with respect to $M$, and
furthermore $(a_i c - ca_i)M = 0$ for each $1 \le i \le \ell$ and
each $c \in C_1$. (In other words, $A/\!\Ann _A M$  is a
$C_1$-algebra in the natural way.) Then  $A/\!\Ann M$ is finite as
a $C_1$-algebra.
\end{thm}

\begin{proof}
Apply Shirshov's height theorem \cite[Theorem~2.3]{BR} to
$A/\!\Ann_A M$.
\end{proof}

Inspired by \cite[Theorem~3.11]{BRV3}, we formulate the following
definition.

\begin{defn}\label{reduct} Suppose $\Gamma$ is the full
quiver of an algebra $A$. A {\bf reduction} of $\Gamma$ is a
pseudo-quiver $\Gamma'$ obtained by at least one of the following
possible procedures:

\begin{enumerate}
\item New relations on the base ring and its pseudo-quiver
$\Gamma'$ obtained by appropriate new gluing. This means:
\begin{itemize}\item Gluing, perhaps up to infinitesimals, with or
without a Frobenius twist (when the gluing is of a block with
itself, with a Frobenius twist, it must become finite); \item New
quasi-linear relations on arrows, perhaps up to infinitesimals;
\item Reducing the matrix degree of a block attached to a vertex.
\end{itemize}

\item New linear dependencies on vertices (which could include
canceling extraneous vertices) between which any two paths must
have the same grade.
\end{enumerate}

A {\bf subdirect reduction} $\{\Gamma'_1, \dots, \Gamma'_m\}$ of
$\Gamma$ is a finite set of reductions of $\Gamma$. A quiver is
{\bf subdirectly irreducible} if it has no proper subdirect
reduction.
\end{defn}

\begin{lem}\label{induc}
Every descending chain of reductions of our original pseudo-quiver
must terminate after a finite number of steps.
\end{lem}
\begin{proof} By definition, any reduction  erases or identifies vertices
or arrows (after sufficiently many new quasi-linear relations), or
lowers the degree vectors of the branches lexicographically, or
lowers their grades (Remark~\ref{dv2}). Each of these processes
must terminate, so the reduction procedure must
terminate.\end{proof}

 This is the key to our discussion of Specht's
problem, since it enables us to formulate proofs by induction on
the reduction of a pseudo-quiver.

\begin{lem}\label{addid} Suppose the algebra $A$ and the polynomial $\bar f$ are as in
\Lref{complem}. Then the full quiver $\Gamma'$ corresponding
to the T-ideal $\id(A)\cup \{ \bar f \}$ is a reduction of the
full quiver $\Gamma$ corresponding to the $\id(A).$
\end{lem}
\begin{proof} By construction, $\bar f$ has nonzero evaluations
along the algebra of $\Gamma$, so the full quiver $\Gamma'$ could
not be $\Gamma$, and thus must be a reduction.\end{proof}

\begin{rem}\label{addid1} In summary, given a T-ideal $\mathcal I$, the Zariski closure $A$
of its relatively free algebra $F\{ x \}/\mathcal I$ has some full
quiver $\Gamma$. Any $A$-admissible non-identity $f$ gives rise to
a symmetrized $\bar q$-characteristic coefficient-absorbing
polynomial $\bar f$, not an identity of $A$. Letting $\mathcal I'
$ be the T-ideal generated by $\mathcal I \cup \{ \bar f \}$, we
see that the full quiver $\Gamma'$ of the Zariski closure $A$ of
the relatively free algebra $F\{ x \}/\mathcal I'$ is a reduction
of $\Gamma.$\end{rem}

\section{Solution of Specht's problem for affine algebras over finite fields}

Our verification of Specht's problem over finite fields
involves an inductive procedure on full quivers. After getting
started, we need each chain of reductions of a full quiver to
terminate after a finite number of steps. To do so, we must cope
with infinitesimals, which appear in~ Definition ~\ref{reduct},
requiring a few observations about Noetherian modules in order to
wrap up the proof. In all of the applications here, the Noetherian
module will be a relatively free associative algebra, but for
future applications to nonassociative algebras we rely only on the
module structure and state these basic observations without
referring to the algebra multiplication. Since Specht's problem is
solved by Kemer in characteristic 0, we assume throughout this
section that the base field has characteristic $p$.

\subsection{Torsion over fields of characteristic $p$}

Throughout this subsection, we assume that $F$ is a field of
characteristic $p$, and $C$ is a commutative Noetherian
$F$-algebra.

\begin{defn}\label{tor0}
Let $M$ be a $C$-module. For any $a\in M$, an element $c\in C$ is
$a$-\textbf{torsion} if there is $k > 0$ such $c^k a = 0$. An
element $c\in C$ is $M$-\textbf{torsion} if it is $a$-torsion for
each $a \in M$. 
We define
$\tor(C)_a = \set{c \in C \,:\, \mbox{$c$ is $a$-torsion}}$.
\end{defn}

\begin{lem}\label{shrink}
 For any finite
$C$-module~$M$ and any $a\in M$, $\tor(C)_a$ is an ideal of $C$.
Furthermore,  define
$\tor(C)_{a;k} = \{ c \in \tor(C)_a : c ^{p^k}a = 0 \}$. Then
$\tor(C)_a = \tor(C)_{a;k}$ for some $k$.\end{lem}
\begin{proof}  $\tor(C)_a$ is an
ideal, since we are in characteristic $p$. Then the series
$\tor(C)_{a;1} \subseteq \tor(C)_{a;2} \subseteq \cdots$
stabilizes, so $\tor(C)_a = \tor(C)_{a;k}$ for some $k$.
\end{proof}

\begin{prop}\label{4.3}
Suppose $\hat A = \hat C \{ a_1, \dots, a_t \}$ is a relatively free, affine algebra over a commutative
Noetherian $F$-algebra $\hat C$. Then $\hat A$ is a finite
subdirect product of an algebra~$\hat A'$ defined over the
$\hat C/{\operatorname{tor}(\hat C) _{a_i}}$, $1 \le i \le t,$ together with
the $\left\{ \hat A/ c ^j \hat A : { c \in \hat C,\ j <k}
\right\}$ where $k $ is the maximum of the torsion indices of
$a_1, \dots, a_t$. 
\end{prop}
\begin{proof}
Let $\hat A_{i}$ denote the direct product of the localizations of $\hat
A$ at the (finitely many) minimal prime ideals of
$\operatorname{Ann} _C a_i$. There is a natural map $$\phi \co \hat A \to \bigoplus \hat A_{i}\oplus \left(\bigoplus _{ c \in \hat C,\ j <k} \hat A/ c ^j \hat A \right).$$

If $a\in \ker \phi$, then looking
at the first component we see that $a$ is annihilated by some
power of some $c\in \hat C$, but this is preserved in one of the
other components of $\phi(\hat A)$; hence, $a = 0$. In other
words, $\phi$ is an injection.
\end{proof}

We quote \cite[Lemma~3.10]{BRV3}, in order to continue:

\medskip {\it Suppose $A$ is a relatively free PI-algebra with
pseudo-quiver $\Gamma$ with respect to a representation $\rho \co
A \mapsto \M[n](C)$, and $\mathcal I = \id(A)$ is a $C$-closed
T-ideal. Then $A$ is PI-equivalent to the algebra of the
pseudo-quiver $\Gamma$.}

\subsection{The main theorem in the field-theoretic case}

We are ready to solve Specht's problem for affine algebras over
finite fields. Let us recall a key result from {\cite[Corollary~4.9]{BR}}.
\begin{prop}\label{Lewin}
Any T-ideal of an affine $F$-algebra contains the set of identities of
some finite dimensional algebra, and thus of $\M[n](F)$ for some $n$.
\end{prop}
(The proof is characteristic free: The radical is nilpotent by the theorem of Braun-Kemer-Razmyslov, so one can display the relatively free algebra $A$ as a homomorphic image of a generalized upper triangular matrix algebra, by a theorem of
Lewin, which satisfies the identities of $n\times n$ matrices.)

\begin{lem}\label{induction}
Suppose $A$ is a relatively free affine algebra in the variety of
a \Zcd\ algebra $B$.

Consider a maximal path in the full quiver of $B$ with the corresponding degree vector~$v_A$. Let $\CJ$ be the ideal generated by the homogeneous elements of the degree vector~$v_A$.
Then $A/\CJ$ is the relatively free algebra of a \Zcd\ algebra,
and hence representable, and its full quiver has fewer maximal
paths of degree $v_A$ than $A$.
\end{lem}
\begin{proof}
The proof is similar to that of the Second Canonization Theorem,
\cite[Theorem~3.7]{BRV3}. Consider a maximal graded component in
$A$. Add characteristic coefficients of the generators of the
generic algebra constructed from $B$, and note that they agree
with the grading of the paths.
 Factoring out the product corresponding to the maximal degree vector
 we obtain a representable
 algebra, $B'$. Construct the full quiver of $B'$ as in the proof of
\cite[Theorem~3.7]{BRV3}.
 Then all paths in $B'$ have fewer maximal paths of degree $v_A$,
and $A/\CJ$ is the relatively free algebra of $B'$.
\end{proof}

 We say that a T-ideal $\CI$ of $F\{ x\}$ is
\textbf{representable} if $F\{ x\}/\CI$ is a representable
algebra.

\begin{thm}\label{Spechtfin} Suppose $A$ is a relatively free,
affine PI-algebra over a field $F$. Then any chain of T-ideals in
the free algebra of $F\{ x\}$ ascending from $\id(A)$ must
terminate.
\end{thm}

\begin{proof} First, we need to move to representable affine
algebras. But, by Proposition~\ref{Lewin}, the T-ideal of $A$
contains the T-ideal of a finite dimensional algebra, so we can
replace $A$ by that algebra.

We want to show that any ascending
chain of T-ideals
\begin{equation}\label{asch} \CI_1 \subseteq \CI_2 \subseteq \CI_3
\subseteq \cdots\end{equation} in the free algebra $F\{x\},$ with
$\CI_1 =\id(A)$, stabilizes. For each $j$, let
$\CI_{j}^{(0)}\subseteq \CI_{j}$ denote the T-ideal of $A$
generated by symmetrized $\bar{q}$-characteristic
coefficient-absorbing polynomials of $\CI_j$ having a non-zero
specialization with maximal degree vector.

Then we get the chain
\begin{equation}\label{asch1} {\CI_1}^{(0)} \subseteq {\CI_2}^{(0)}
\subseteq {\CI_3}^{(0)} \subseteq \cdots .\end{equation}

Let $\Gamma$ be the full quiver of $A$, so $\id(A) = \id(\Gamma).$
Let $v_A$ denote the maximal degree vector for a non-zero
evaluation of some polynomial in $A$. Since $\CI _{j}^{(0)}$
 is Noetherian, by Theorem~\ref{Shir} applied to Theorem~\ref{tradj}, we can define
$\CI_j^{(1)}$ to be the maximal T-ideal of $\hat A$ contained in
$\CI_j$.

Then, the chain
\begin{equation}\label{asch11}
{\CI_1}^{(1)} \subseteq {\CI_2}^{(1)} \subseteq {\CI_3}^{(1)}
\subseteq \cdots
\end{equation}
of ideals is in the Noetherian algebra $\hat A$, and thus stabilizes
at some $\CI_{j_0}^{(1)}$.

Passing to $A/\CI_{j_0}^{(1)}$, we may assume that $\CI_{j}^{(1)}=
0$ for each $j > j_0$. Hence, $A/\CI_{j_0}^{(1)}\subseteq
\hat{A}/\CI_{j_0}^{(1)}$ is representable. If $\CI_j^{(0)}$ were
nonzero then the fully hiked polynomial of some $0 \neq f \in
\CI_j^{(0)}$ would be in $\CI_j^{(1)}=0$, a contradiction. Thus,
$\CI_j^{(0)} = 0$ for each $j > j_0$. In other words, $\CI_j$ has
only zero evaluations in degree $v_A$.

Finally, let $\CJ$ be the T-ideal defined in
\Lref{induction}. Thus
$\CJ \cap \CI_j = \CI_j^{(0)}$, so passing to $A/\CJ$, which is
relatively free and representable by \Lref{induction}, we
lower the maximum degree vector, and conclude by induction.

But these lift to a chain of ideals of the Noetherian algebra
$\hat A$, which must stabilize, and thus \eqref{asch11} stabilizes
at some $\CI _{j_0}^{(0)}$. Passing to $A/\CI _{j_0}^{(0)}$ and
starting the chain at $j_0$, we may assume that $\CI _{j}^{(0)}=
0$ for each $j$.

We let $v_\mathcal B$ be the degree vector of some $A$-admissible
polynomial, and let $\CI_\mathcal B$ be the T-ideal of
Theorem~\ref{traceq2}, which shows that $\CI _{j}\cap \CI_\mathcal
B = 0$ for each $j> j_0.$ Let $ \overline{\CI_j} = (\CI _{j} +
\CI_\mathcal B)/\CI_\mathcal B,$ a T-ideal of the relatively free
algebra $ F\{ x\}/ \CI_\mathcal B$ for each $j>j_0.$ By induction
on the maximal degree vector of the quiver, applied to the
relatively free algebra $ F\{ x\}/ \CI_\mathcal B,$ the chain of
T-ideals
\begin{equation}\label{asch2} \overline{\CI_{j+1}} \subseteq
\overline{\CI_{j+2}}\subseteq \overline{\CI_{j+3}} \subseteq
\cdots\end{equation} must stabilize, so we conclude
that the original chain of T-ideals must stabilize.
\end{proof}

\section{Solution of Specht's problem for PI-proper T-ideals of  affine algebras over arbitrary commutative Noetherian rings}\label{nonassoc}

Using the same ideas, we finally can prove Specht's problem for
affine PI-algebras over a commutative Noetherian ring $C$. Our strategy is to reduce to algebras over fields, since this
 case is already solved. The argument
is based on a formal reduction from algebras over rings to algebras over fields, much of which
we formulate rather generally for algebras which are not necessarily
associative. Accordingly, we fix a given algebraic variety $\mathcal V$ of algebras, and
consider Specht's problem for algebras in $\mathcal V$. When necessary,
we take $\mathcal V$ to be the variety of associative algebras.

 One can construct the free nonassociative algebra,
whose elements are polynomials where we also write in parentheses
to indicate the order of multiplication. The
\textbf{T-ideal} of a set of polynomials \textbf{in an algebra}
$A$ is the ideal generated by all substitutions of these
polynomials in $A$. The variety $\mathcal V$ itself is defined by
a T-ideal $\assoc$ of identities of the free nonassociative algebra, and the corresponding
factor algebra is the relatively free algebra of the variety $\mathcal V$. (For example, when $\mathcal V$ is the variety of associative algebras,
$\assoc$ is generated by the \textbf{associator} $(x_1x_2)x_3 - x_1(x_2 x_3).$)
Specht's problem now is whether every chain of ascending T-ideals containing the associator stabilizes, or, equivalently, if every T-ideal is finitely based modulo the T-ideal  $\mathcal I_{\operatorname{assoc}}$ of the associator.

 There
is an extra technical issue, since usually the definition of PI
requires that the ideal of the base ring generated by the
coefficients of the PIs is all of $C$. (This is obviously the case
when $C$ is a field). We call such a T-ideal {\it PI-proper}, and
start in this section with that case. Finally, in \S\ref{notprop}
we prove the general result for T-ideals which are not necessarily
PI-proper.

 \subsection{Reduction to algebras over integral domains}

The considerations of this section apply to arbitrary algebraic varieties (not necessarily associative).

\begin{defn}\label{Spechtobs22}
 We say that a Noetherian ring $C$ is $\mathcal V$-\textbf{Specht} if Specht's problem has a positive solution
in the variety $\mathcal V$ for PI-proper T-ideals defined over $C$, i.e., any PI-proper
T-ideal generated by polynomials $f_1, f_2, \dots$ is finitely
based  modulo  $\assoc$.
 We say that $C$ is \textbf{almost $ \mathcal V$-Specht} if $C/I$ is $\mathcal V$-Specht for
every nonzero ideal $I$ of $C$. If~$\mathcal V$ is not specified here, it is assumed to be the variety of associative algebras.
\end{defn}

\begin{rem}\label{Spechtobs2} Here is the general reduction of Specht's problem to the case when $C$ is an integral
domain. We need to show that any T-ideal generated by polynomials
$f_1, f_2, \dots$ is finitely based modulo $\assoc$. Let $\CI_{j}$ be the T-ideal generated by $f_1,\dots,f_j$. By
Noetherian induction, we may assume that $C$ is almost $\mathcal
V$-Specht. Suppose $c_1c_2 = 0$ for $0 \ne c_1, c_2 \in C.$
By Noetherian induction, the system $\{f_j\}$ is finitely based
over $c_2 A$, so there is some $j_0$ for which each $f_j=g_j+c_2
h_j$ where $g_j\in\CI_{j_0}$ and $h_j$ is arbitrary. The polynomials $f_j$ can be
replaced by $c_2 h_j$ for all $j>j_0$. But the T-ideal generated
by $\{c_2 h_j :{j>j_0}\}$ in $A/c_1 A$ is finitely based, by
Noetherian induction over~$C/c_1C$.

Thus, we are done unless $c_1c_2 \ne 0$ for all  $0 \ne c_1, c_2 \in C,$
implying that $C$ is an integral domain.
\end{rem}

\subsection{Reduction to prime power torsion}

By Remark~\ref{Spechtobs2}, we may assume from now on that
$C$ is an almost $\mathcal V$-Specht integral domain. Also, we
assume that $C$ is infinite, since otherwise $C$ is a field, and
we have solved the case of fields. To proceed further, we also
introduce torsion in
 the opposite direction.

\begin{defn}\label{tor00}
Let $M$ be a module over a commutative integral domain~$C$.
For $z \in C$ define $\Ann_M(z) = \{ a \in M: z a = 0\}$.

  $\Ann_M(z) $ is called the  $z$-\textbf{torsion} of $M$.
$M$ is $z$-\textbf{torsionfree} if $\Ann_M(z) =  0$.  For $I \triangleleft C$ and a $C$-module $M$, define $\tor _I(M) = \cup _{0 \ne z \in C} \Ann_M(z) , $
a $C$-submodule of $M$ since $C$ is an integral domain.

The {\bf
$z$-torsion index} of $M$ is $k$ if the chain $$\Ann_M(z) \sub
\Ann_M(z^2) \sub \cdots \sub \Ann_M(z^k) \sub \cdots$$ stabilizes
at $k$. Reversing \Dref{tor0}, we say that $M$ is
$z'$-\textbf{torsionfree} if its $w$-torsion submodule is 0 for
each prime $w$ not dividing $z$.
\end{defn}

\begin{lem}\label{Spechtobs} If $A$ is a relatively free algebra and $c\in C$, then
$\Ann_A c $ is a T-ideal, with $cA \cong
A/\!\Ann_A c$ as $C$-modules (but not necessarily as
$C$-algebras).
 We have an inclusion-reversing map $\{$Ideals of $C\} \to \{$T-Ideals of~$A\}$
given 
by $I \mapsto \operatorname{tor}_I(A).
$
\end{lem}
\begin{proof} $\Ann_A c$ is clearly a T-ideal, by \Lref{Tid}, and the
rest of the first assertion is standard. The second assertion is
likewise clear.
\end{proof}

In any commutative Noetherian domain, any element can be factored as
a finite product of irreducible elements (although not necessarily
uniquely). 

\begin{prop}\label{subdir1}
Any  module over an integral domain $C$ whose torsion
involves only finitely many irreducible elements is a subdirect product of
finitely many $z'$-torsionfree modules, where $z$ ranges over
these irreducible elements of $C$.
\end{prop}
\begin{proof}
Follows at once from the lemma.
\end{proof}

A closely related result, which we record for reference in future
work:
\begin{prop}\label{subdir2S}
Any Noetherian $\Z$-module is a subdirect product of finitely many
$p'$-torsionfree modules, where $p$ ranges over the prime
numbers.\end{prop}

\subsection{Homogeneous T-ideals}

Recall that a polynomial is {\bf homogeneous} if for each
indeterminates $x_i$ each of its monomials has the same degree in
the indeterminate $x_i$. Since the degrees grade the free algebra,
any polynomial has a unique decomposition as a sum of homogeneous
polynomials, which we call its {\bf homogeneous components}.
Recall that a T-ideal $\CI$ is {\bf{homogeneous}} if it contains
all of the homogeneous components of each of its polynomials.
\begin{rem}
Although any homogeneous T-ideal is clearly generated by homogeneous polynomials, in general, homogeneous polynomials need not generate a homogeneous T-ideal, because of the vagaries of the quasi-linearization procedure (see \cite[Example~2.2]{BRV4} and \cite[Exercise~13.10]{BR}).
\end{rem}

\begin{defn}\label{5.8}
A set of polynomials $S$ is {\bf{ultra-homogeneous}} if it contains the homogeneous
component of every element in $S$, as well as of the quasi-linearizations of all polynomials in $S$.

The {\bf ultra-homogeneous closure}
$\overline{S}_{\operatorname{uh}}$ of a set $S$ is the intersection of all ultra-homogeneous sets containing it. (The ultra-homogeneous closure of a finite set is finite, since the procedure terminates for
each polynomial after finitely many steps.)

The {\bf homogeneous
socle} $\underline{\mathcal{\CI}}_{\operatorname{soc}}$ of a T-ideal $\CI$ is  the union of homogeneous T-ideals contained in $\CI$.
\end{defn}

Note that any homogeneous T-ideal $\CI$ contains the homogeneous components
of the quasi-linearizations of each of its polynomials, so is automatically ultra-homogeneous.

\begin{prop}\label{5.80}
The T-ideal generated by an ultra-homogeneous set of polynomials
$S$ is homogeneous.
\end{prop} \begin{proof}
We need to show that the homogeneous components of any
substitution remain in the T-ideal $\CI$ generated by
$S = \overline{S}_{\operatorname{uh}}$. By definition of
quasi-linearization, it is enough to check this for monomial
substitutions. But these are specializations of substitutions of
letters (taking a different letter for each monomial), and thus
are specializations of the homogeneous components of the
quasi-linearizations, which by definition are in~$\CI$.
\end{proof}

In particular
every set of multilinear identities generate a homogenous
T-ideal.

\begin{cor}\label{quasi10}
Let ${\mathcal V}$ be a variety satisfying the ACC on homogeneous T-ideals.
Then every homogeneous $T$-ideal is finitely based.
\end{cor}
\begin{proof}The ultrahomogeneous closure of a finite set of
polynomials is finite.
\end{proof}

Let $z \in C$ be any nonzero element. Our overall goal would be to
prove formally that if every field is $\mathcal V$-Specht then
every commutative Noetherian ring  is $\mathcal V$-Specht.
Unfortunately, this is not quite in our grasp, since one detail
still relies on associativity. We can prove the following theorem:

\begin{thm}\label{firstred}
Let $\mathcal V$ be a variety of algebras such that every field is
$\mathcal V$-Specht. If an integral domain $C$ is almost $\mathcal
V$-Specht, and $\mathcal V$-Specht with respect to homogeneous
T-ideals, then $C$ is $\mathcal V$-Specht.
\end{thm}

Taking $\mathcal V$ to be the class of associative algebras, we
conclude by proving that if a Noetherian ring $C$ is almost
Specht and every field is Specht, then $C$ is Specht with
respect to homogeneous T-ideals.

This will affirm Specht's problem for  affine
PI-algebras over an arbitrary Noetherian ring, and together with
Theorem~\ref{SpechtNoeth1} below will affirm Specht's problem for
arbitrary  affine algebras over a  Noetherian ring.

We deal with the reduction for other varieties in a subsequent paper.

\subsubsection{Proof of Theorem~\ref{firstred}}

Although we are working in the context of associative algebras, the  proof of~Theorem~\ref{firstred}  also works
analogously for nonassociative algebras.

\begin{lem}\label{quasi1Lem} Suppose
$\CI$ is a T-ideal, and $f = \sum f_i\in \CI$ has total degree
$n$ (where $f_i$ are the homogeneous components). Then for every
Vandermonde determinant $d$ of order $n$,
$d \overline{\langle f\rangle}_{\operatorname{uh}} \subseteq \CI$.\end{lem}

\begin{proof} Substitute $\lambda_{j} x_j$ for $x_j$, for fixed $j$.
Let $d$ denote the determinant of the Vandermonde matrix
$(\lambda_{j}^i)$ and $i = 1,\dots,n$. The homogeneous components
of $d f$ are in $\CI$, by the usual Vandermonde argument of
multiplying by the adjoint matrix.
\end{proof}

\begin{lem}\label{first1} Suppose $C$ is
$\mathcal V$-Specht with respect to homogeneous T-ideals, and $\mathcal I $ is a proper T-ideal that properly contains
its homogeneous   socle   $\mathcal I_0 .$ Then $\mathcal I $
contains a homogeneous T-ideal of the form $
d\mathcal I_1 ,$ where $\mathcal I_1\supset \mathcal I_0$ is a
finitely based, proper T-ideal. (Here $d$ is a product of
Vandermonde determinants.) \end{lem}
\begin{proof} Take a proper polynomial $f\in \mathcal I \setminus  \mathcal I_0$.
By Lemma~\ref{quasi1Lem} there
is $0 \ne d_1 \in C$ such that $d_1 f_i \in \mathcal I $, where
$f_i$ are the homogeneous components of the quasi-linearizations
of $f$. Continuing with the quasi-linearization procedure, which
is finite, we see by induction that there is some $d'$ such that
$d' g_{i,j} \in \mathcal I $, for each component $g_{i,j} $ in the
various quasi-linearizations, implying $d'd_1 \overline{\langle f
\rangle}_{\operatorname{uh}} \subseteq \mathcal I$, as desired.
\end{proof}

Note that we had to take the ultra-homogeneous closure of
a polynomial, and not a T-ideal, to remain with a finite set of polynomials.
We are ready to prove Theorem~\ref{firstred}.
\begin{proof}
Let $A$ be an algebra in $\mathcal V$ . We introduce some notation: Given $z\in C$ and a T-ideal $\Gamma $, we write
$\CI_\Gamma(z)$ for the kernel of the composite map $A \to zA \to
zA/(\Gamma \cap zA)$. 
 By hypothesis we can take $z\in C$ with
$\overline{\CI_\gamma(z)}_{\operatorname{uh}} $ maximal, and we
write $z_\Gamma $ for $z$ and $\CJ(\Gamma)$ for
$\CI_\gamma(z)$. $\CJ(\Gamma) =
\underline{\CJ(\Gamma)}_{\operatorname{soc}}$, since otherwise we
could use Lemma \ref{first1} to increase $\underline{\CJ(\Gamma)}_{\operatorname{soc}}$,
contrary to its definition. 
Hence, $\CJ(\Gamma)$ is already homogeneous.\medskip

Given a chain $\Gamma_1 \subseteq \Gamma_2 \subseteq \cdots$ of
T-ideals of $A,$ we see by the hypothesis on homogeneous
T-ideals that there is $i $ such that $\CJ(\Gamma_j) =
\CJ(\Gamma_i)$ for all $j\ge i.$ Write $\hat z = \prod
_{j\le i} z_{\Gamma_j}$. Then the $\CJ(\Gamma_j) \cap \hat
z A$ also stabilize. If some $\Gamma_j \cap \hat z A$ properly
contains $\CJ(\Gamma_i) \cap \hat z A$ then it has a
(nonhomogeneous) polynomial $f$ and thus contains $\hat z
\overline{f}_{\operatorname{uh}}$, which is impossible unless
$\hat z C$ is a proper ideal of $C$.

But $\mathcal I_j(A)/\hat z \mathcal I_j(A)$ is a T-ideal over
$C/\hat z C$, so we conclude by Noetherian induction.
\end{proof}

\bigskip

\subsection{Conclusion of the solution of Specht's problem for arbitary affine
PI-algebras over Noetherian rings}
$ $

We start with some general considerations that can be used for arbitrary varieties.
The following well-known fact is a key ingredient, yielding a tool
for applying Noetherian induction.
\begin{lem}[Baby Fitting Lemma]\label{lem3} Let $M$ be a $C$-module, with $z\in C,$ and take any $k\in \N$. 
Suppose 
$\Ann_M(z^{k+1}) \sub \Ann_M(z^k)$. Then $z^k M \cap \Ann_M(z) =
0$.
\end{lem}
\begin{proof}
If $z^k a \in \Ann_M(z)$, then $z^{k+1}a = 0$, implying $z^k a =
0$ by assumption.
\end{proof}

We also need some easy facts from module theory.

\begin{lem}\label{easy1}
Let $M,N$ be modules over a commutative ring $C$. Let $f \co M \ra N$ be a homomorphism of modules.

(i) For every $z,z' \in C$, if the induced homomorphisms $f' \co
M/z'M \ra N/z'N$ and $f'' \co z'M/zz'M \ra z'N/zz'N$ are 1:1, then
so is the induced homomorphism $$f''' \co M/zz'M \ra N/zz'N.$$

(ii) If the induced homomorphisms $z^iM/z^{i+1}M \ra z^iN/z^{i+1}N$ are 1:1 for every $0\leq i < k$, then the induced homomorphism $M/z^{k}M \ra N/z^{k}N$ is 1:1 as well.
\end{lem}
\begin{proof}
(i) If $a \in \ker f''',$ then $a +z'M \in \ker f' = 0,$ implying
$a \in \ker f'' = 0$.

(ii) Induction on $k$, taking $z' = z^{k-1}$ in the previous lemma.

\end{proof}

Let $S$ denote the monoid generated in $C$ by $z$. Recall that
$$S^{-1}M = \set{s^{-1}a \suchthat s \in S, a \in M},$$ where
$s^{-1}a=s'^{-1}a'$ if there is $s_0 \in S$ such that
$s_0(s'a-sa')=0$. In particular $s^{-1}a=0$ if there is $s_0 \in S$ such that $s_0a=0$.
\begin{lem}\label{L3}
Let $f \co M \ra N$ be a homomorphism of $C$-modules, and let $z \in C$. Let $f' \co S^{-1}M \ra S^{-1}N$ and $f_i \co M/z^iM\ra N/z^{i}N$ be the induced homomorphisms, where $S$ is the monoid generated by $z$.
 \begin{enumerate}
\item Assume $N$ has $z$-torsion index $k$.
If every $f_i$ is one-to-one and $f'$ is onto, then the restriction $f|_{z^kM} \co {z^kM} \to {z^kN}$ is onto.
\item\label{L3.2} Assume $M$ has $z$-torsion index $k$. If $f'$ is one-to-one, then the restriction $f|_{z^kM} \co {z^kM} \to {z^kN}$ is one-to-one.
\item\label{L3.3} Assume $M$ has $z$-torsion index $k$. If $f_k$ and $f'$ are one-to-one, then $f$ is one-to-one.
\end{enumerate}
\end{lem}
\begin{proof}
\begin{enumerate}
\item Let $b \in N$. By assumption there is an element $z^{-\ell}a \in S^{-1}M$ such that $z^{-\ell}f(a) = f'(z^{-\ell}a) = 1^{-1}b \in S^{-1}N$, so for some $\ell'\geq 0$ we have that $z^{\ell'}f(a) = z^{\ell+\ell'}b$. Then $f_{\ell+\ell'}(z^{\ell'}a+z^{\ell+\ell'}M) = f(z^{\ell'}a)+z^{\ell+\ell'}N = 0$, so $z^{\ell'}a = z^{\ell+\ell'}a'$ for some $a' \in M$. But now $z^{\ell+\ell'}b = f(z^{\ell'}a) = z^{\ell+\ell'}f(a')$, so $z^{\ell+\ell'}(b-f(a')) = 0$. Since the torsion index of $N$ is $k$, we have that $z^kb = f(z^ka')$.
\item Let $a \in M$ be such that $f(z^ka)=0$. Then $f'(1^{-1}z^ka) =
1^{-1}z^kf(a) = 0$ in~$S^{-1}N$, so by assumption $1^{-1}z^ka=0$,
namely for some $\ell \geq 0$, $z^{k+\ell}a = 0$. Since $k$ is the
$z$-torsion index of $M$, we have that $z^ka = 0$.
\item Let $a \in M$ be such that $f(a) = 0$. Then $f_k(a+z^kM) = 0+z^kN$, so $a \in z^k M$, but then \eq{L3.2} implies that $a=0$.
\end{enumerate}
\end{proof}

Finite torsion index is essential in Lemma~\ref{L3} (which is why
Lemma~\ref{up} below only applies to homogeneous T-ideals).
\begin{exmpl}
(An example where the restrictions $f'$ and $f_i$ are
isomorphisms, but $f|_{z^kM} \co {z^kM} \to {z^kN}$ is neither
onto nor one-to-one). Let $C = F[z]$, and $P = F[[z^{-1}]]/F$ with
the natural $C$-module structure. Since multiplication by $z$ is
onto, $P/z^i P = 0$ for every $i$, and $S^{-1}P = 0$ since
$1^{-1}(z^{-i}) = z^{-i}z^i(z^{-i}) = z^{-i}0 = 1^{-1}0$. Let $f$
be the zero map from $P \oplus 0$ to $0 \oplus P$; it is neither
one-to-one nor onto, but the induced maps $f_i$ and $f'$ are
clearly (trivial) isomorphisms. Indeed, $P$ has infinite
$z$-torsion index.
\end{exmpl}

\begin{rem} For any $z \in C, $ $zA \cong A/\!\Ann z$ is a T-ideal of
$A$.\end{rem}

To progress with the proof over an arbitrary base ring, we first need the special case where
the T-ideal contains a representable T-ideal.

\begin{thm}[Small Specht Theorem]\label{SpechtNoeth51}
Let $C$ be an almost Specht, commutative Noetherian ring,
 and $A$ ~an affine PI-algebra containing a representable
T-ideal $\CI$; i.e., the algebra $A/\CI$ is representable. 
 Then any chain of T-ideals in the free algebra of~
$C\{ x\}$ ascending from $\id(A)$ stabilizes. 
\end{thm}
\begin{proof}
By Lemma~\ref{Spechtobs2}, $C$ is an integral domain. We need to
show that any ascending chain of PI-proper T-ideals
\begin{equation}\label{asch3} \CI_1 \sub \CI_2 \sub \CI_3 \sub
\cdots\end{equation} of $A$, stabilizes. Since $\CI \subseteq
\CI_1$, we may replace $A$ by $A/\CI, $ and assume that $A$ is
representable. We view $A \subseteq M_n(K),$ where $K$ is an
algebraically closed field containing $C$.  If $C$ is finite, then
it is a field, and we are done by Theorem~\ref{Spechtfin}. So we
may assume that $C$ is an infinite integral domain. Denoting $AK$
as $A_K$, we work with respect to a quiver $\Gamma$ of $A_K$ as a
$K$-algebra.

As in \Tref{Spechtfin}, let $\CI_j^{(1)}$ be the maximal subideal
of $\CI_j$ closed under multiplication by $\hat{C}$ of
Theorem~\ref{tradj}. 
Thus
\begin{equation}\label{asch31} \CI_1^{(1)} \sub \CI_2^{(1)} \sub \CI_3^{(1)} \sub \cdots\end{equation}
are ideals in the Noetherian algebra $\hat{A} = \hat{C}A$, so this chain stabilizes, and we may assume $\CI_{j}^{(1)} = \CI_{j_0}^{(1)}$ for $j > j_0$.

For a $T$-ideal $\CI$ of $A$, let $\overline{\CI} = K  \CI$, taken
in $A_ K$. Define $\widetilde {\CI} =K \CI \cap A \supseteq \CI$.
Let $A' = A/{\CI_{j_0}^{(1)}}$. Passing down to  $A'$, we shall
pass further to $A/\widetilde{\CI_{j_0}^{(1)}}$.

The quotient $\widetilde{\CI_{j_0}^{(1)}}/{\CI_{j_0}^{(1)}}$  is
torsion, so there is $0 \neq z \in \hat{C}$ such that
$z \widetilde{ \CI_j^{(1)}} = z \widetilde{ \CI_{j_0}^{(1)}} \sub
\CI_{j_0}^{(1)}$. The chain $\Ann _{\hat{A} } z  \sub
\Ann_{\hat{A} }  z^2 \sub \cdots \sub \Ann _{\hat{A} } z^k \sub
\cdots$ stabilizes at some $k$, by the Noetherianity of $\hat A$.
Now, applying the baby Artin-Rees lemma to
$\hat{A}/\CI_{j_0}^{(1)}$, we see that
$$z^{k}\hat{A}  \cap \widetilde{\CI_j^{(1)}} \sub  \CI_{j_0}^{(1)}.$$

In particular the natural map  $$A'\ra (A'/z^{k}A') \,\oplus\,
(A/\widetilde{\CI_{j_0}^{(1)}})$$ is an injection. The image of
the chain~\eqref{asch3} of  the first summand on the right
stabilizes by applying Noetherian induction. Thus, we pass to the
second summand of the right, which has no $C$-torsion. Letting
$\CJ$ be the ideal constructed in \Lref{induction}, we have for
every $j > j_0$ that $\CI_{j} \cap \CJ= 0$ in
$A_K/A_K\widetilde{\CI_{j_0}^{(1)}}$ as in the last paragraph of
the proof of Theorem~\ref{Spechtfin}. Hence, a fortiori, $\CI_{j}
\cap \CJ= 0$ in $A/\widetilde{\CI_{j_0}}$. We are done by
induction on the degree vector.
\end{proof}

\begin{lem}\label{up}
Suppose $z \in C$ such
that $C/zC$ and $C[z^{-1}]$ are $\mathcal V$-Specht. Then $C$
satisfies the ACC on homogeneous T-ideals from $\mathcal V$.
\end{lem}
\begin{proof} By induction on the length of $z$ as a product of primes, we may assume that $z$ is prime.
Let $A$ be an affine algebra over $C$, and let $$\CI_1 \sub \CI_2 \sub \CI_3 \sub \cdots$$
be an ascending chain of T-ideals in $A$.

Let $A_i = A/\CI_i$, and consider the infinite commutative diagram
\begin{equation}
\xymatrix@R=15pt@C=12pt{ %
A_1/zA_1 \ar[r] \ar[d] & zA_1/z^2A_1 \ar[r] \ar[d] & z^2A_1/z^3A_1 \ar[r] \ar[d] & z^3A_1/z^4A_1 \ar[r] \ar[d] & \cdots \\
A_2/zA_2 \ar[r] \ar[d] & zA_2/z^2A_2 \ar[r] \ar[d] & z^2A_2/z^3A_2 \ar[r] \ar[d] & z^3A_2/z^4A_2 \ar[r] \ar[d] & \cdots \\
A_3/zA_3 \ar[r] \ar[d] & zA_3/z^2A_3 \ar[r] \ar[d] & z^2A_3/z^3A_3 \ar[r] \ar[d] & z^3A_3/z^4A_3 \ar[r] \ar[d] & \cdots \\
\vdots & \vdots & \vdots & \vdots & \\
}
\end{equation}
where the left-to-right maps are multiplication by $z$, and the top-to-bottom arrows are the natural projections. So all the maps are projections.

We claim that outside a certain rectangle, all the maps in this
infinite matrix are one-to-one. Indeed, the entries are algebras
over $C/zC$, so each row stabilizes by assumption. Letting $B_i$
denote the final algebra in row $i$, we obtain a chain of
projections $B_1 \ra B_2 \ra \cdots$ which must also stabilize,
proving that for some $k_0$, all of the rows stabilize after $k_0$
steps. We are done since each of the first $k_0$ columns
stabilizes. It follows that when $i$ is large enough, all the maps $z^jA_i/ z^{j+1}A_{i} \ra z^jA_{i+1}/z^{j+1}A_{i+1}$ are isomorphisms, so by \Lref{easy1}.(ii), the natural projection $A_i \ra A_{i+1}$ induces isomorphisms $A_i/z^{j}A_{i} \ra A_{i+1} /z^{j}A_{i+1}$ for every $j$.

Similarly, the chain of projections $$C[z^{-1}]A_1 \ra C[z^{-1}]A_2 \ra \cdots$$ stabilizes by the assumption on $C[z^{-1}]$.

Let $i$ be large enough. Since $\CI_i \sub \CI_{i+1}$ are
homogeneous, 
the
natural projection $A_i \ra A_{i+1}$ preserves the degree grading.
Each homogeneous component is finite as a $C$-module since $A$ is
affine, and thus Noetherian, and therefore has finite $z$-torsion
index. By \Lref{L3}.\eq{L3.3}, the map in each component is
one-to-one, proving that $\CI_i = \CI_{i+1}$.
\end{proof}

The main idea in the proof given above is a simple version of a spectral sequence.
Having proved a special case, we do a more general case (in fact, our most general version holds for
an arbitrary variety $\mathcal V$ of algebras).

\begin{thm}\label{SpechtNoethspec1}
Suppose the relatively free algebra $A$ with respect to a T-ideal $\CI$ is $z'$-torsionfree
for some $z\in C$, where $\CI$ is generated by a polynomial all of
whose coefficients are $\pm 1$. Then any increasing chain of
T-ideals of $A$ starting with $\CI$ must terminate,  for any
commutative Noetherian ring $C$.
\end{thm}

\begin{proof}

Writing $z$ as a product of primes, we may assume that $z$ is prime.
 Let $C_0$ be the
subring of $C$ generated by 1.
  Letting $L$ be the field of fractions of
 $C_0/ (C_0 \cap zC),$ we pass to $A \otimes _{C_0} L,$ so
 we may assume that $C_0$ is a field.
(If there is no $p$-torsion at any step, then we can localize at
the natural numbers and reduce to the case of $\Q$-algebras, which
was solved by Kemer.) But $C_0[z]$ thus is a PID,
and by the argument in Lemma~\ref{up}, we can break up our chain into
 chains of T-ideals defined over $C/zC$, so
 we conclude
 by Lemma~\ref{up}.
\end{proof}

So far, these arguments have been applied to arbitrary varieties,
and in fact there are Lie, alternative and Jordan versions of
Iltyakov \cite{I1,I2} and Vais and Zelmanov~\cite{VZ}; 
their proofs are rather delicate, in part because it is still
unknown whether any alternative, Lie, or Jordan algebra satisfying
a Capelli system of identities must satisfy all the identities of
a finite dimensional algebra. Belov \cite{B2}  obtained a version of the Small Specht Theorem (Theorem~\ref{SpechtNoeth51})  for classes of algebras of characteristic 0 asymptotically close to associative algebras;
this includes alternative and Jordan algebras.

Our method here is to develop some theory to take care of torsion in
polynomials, to conclude the proof of Theorem \ref{SpechtNoeth} below.
 In other words, we need some local-global correspondence that will
enable us to pass from the global situation with torsion to the local situation without torsion. Our main tool is
Proposition~\ref{subdir1}, but this only enables us to handle a finite number of irreducible elements of $C$ producing torsion,
whereas there might be an infinite number of such elements.
Thus we need some way of cutting down from infinite to finite.

The most direct argument   relies on an (associative) result only
available in Russian. Procesi asked whether the kernel of the
canonical homomorphism $\id(\M(\Z)) \to \id(\M(\Z /p\Z))$ is equal
to $p \id(\M(\Z))$.

Schelter and later
Kemer~\cite{K7} provided  counterexamples, but
Samo\u{\i}lov~\cite{Sam} 
showed that  if $p > 2d$, the kernel of the canonical homomorphism
$\id(\M(\Z)) \to \id(\M(\Z /p\Z))$ is indeed equal to $p \id(\M(\Z))$. Unfortunately,
this result appears so far only in his doctoral dissertation ~\cite{Sam}.

Thus, we give two versions for the conclusion of the proof of the
next theorem, the first relying on Samo\u{\i}lov's Theorem, and
the second for those readers who would prefer a full proof of
Theorem~\ref{SpechtNoeth} in English. Another advantage of the second proof
is that its reduction argument works for arbitrary varieties.

\begin{thm}\label{SpechtNoeth}
Any PI-proper T-ideal $\CI$ of $C\{x_1, \dots, x_\ell\}$ is finitely based, for any
commutative Noetherian ring $C$.
\end{thm}

\begin{proof}
Let $A$ be the relatively free algebra of $\CI$. We can replace $\CI$ by the T-ideal of a
PI-proper polynomial $f$ contained in it. But by
Amitsur~\cite[Theorem~3.38]{BR}, any PI-algebra satisfies a power
of a standard polynomial, so we may assume $f$ is such a
polynomial, and thus has all nonzero coefficients in $\{\pm 1\}$.

Consider the localization $A \mapsto A\otimes_\Z \Z/p\Z $, viewed
as a $C_p / p C_p$-algebra.
By a theorem of Bergman and Dicks~\cite{BerD}, there is a canonical homomorphism of $A$ to a
representable algebra, whose kernel $M$, in view of Lewin's theorem, vanishes modulo $p$ for
any prime $p$, i.e., when we map $A \mapsto A\otimes_\Z \Z/p\Z $, viewed
as a $C_p / p C_p$-algebra.
But the map $A \ra A\otimes_\Z \Z/p\Z $ 
is faithful whenever the kernel of
the canonical homomorphism $\id(\M(\Z)) \to \id(\M(\Z /p\Z))$ is
equal to $p \id(\M(\Z))$, where $d$ is the size of matrices in the
representation, which by
Samo\u{\i}lov's Theorem~\cite{Sam} 
happens when $p>2n$, showing that $A$ is $(2n)!$-torsionfree.
The claim then follows by Theorem~\ref{SpechtNoethspec1}.
\end{proof}

We turn now to the second proof.

\begin{rem}\label{pat}
The point here is that in view of  Proposition~\ref{Zub},  taking
$M$ to be the T-ideal generated by $f$ as notated there, assuming that $A$ satisfies
a Capelli identity $c_{n+1}$ and we are in a matrix component of degree $n$, the relatively free algebra $A$
is integral over the affine $C$-algebra $C[\xi]$ where $\xi$ denotes the set of characteristic coefficients
formally corresponding to the finitely many $\delta$ operators. Unfortunately, this case
can be assured only when $C$ is a field, but by a careful use of localization we can formulate
a local-global framework in which we can utilize this situation.
\end{rem}

{\it Second proof  of Theorem~\ref{SpechtNoeth}}.  Let $A = C\{ a_1, \dots, a_\ell\} $ be the relatively free algebra of $\CI$.  We formulate an inductive argument in analogy
to Theorem~\ref{Spechtfin}. In~order to apply the theory of full
quivers, we need to pass to some field. By Remark~\ref{Spechtobs2}
we may assume that $C$ is an integral domain. Let $F$ denote the
field of fractions of $C$, and let $A_F : = A \otimes_C F$. We consider $\CI \otimes F$ which contains the  ideal of identities of $A_F$.
Unfortunately, $(\CI \otimes F) \cap C\set{x}$
might properly contain $\CI.$ If the torsion over $C$
only involved finitely many primes we could handle this by means
of Proposition~\ref{subdir1}, but this need not be the case. Thus,
we need a more delicate argument which enables us to relate $\CI$
with $\CI_F$.

{\it Step 1.} We start with a proper PI of $A$. As mentioned in
the first proof, Amitsur~\cite[Theorem~3.38]{BR} says that every
PI-algebra satisfies some power of a standard identity, which we
denote here as $f$. Let $\CI_0$ denote the $T$-ideal of $ C\{ x\}$
generated by $f$, contained in~$\CI,$ so $\CI_{0} \otimes F$ is the
$T$-ideal of $ F\{ x\}$ generated by $f$, contained in~$\CI \otimes F$.
The relatively free algebra $F\{ x\}/(\CI_{0} \otimes F)$ has some full
quiver $\Gamma_1$. Although $\Gamma_1$ does not have much to do
with the original algebra $A$, it provides a base for an inductive
argument, as well as a handle for using our field-theoretic
results. Since the chain of reductions of any full quiver must
terminate after a finite number of steps, we induct on $\Gamma_1$.

We need to show that every chain $\mathcal C = \set{\CI_1 \sub \CI_2 \sub \CI_3 \sub \cdots}$ of $T$-ideals
ascending from $\CI_0$ stabilizes. Over the field $F$, we could do
this by the argument of Theorem~\ref{Spechtfin}, which we recall
is achieved by hiking $f$, obtaining matrix characteristic
coefficients for the evaluations of a maximal branch of
$\Gamma_1$, redefining these in terms of elements of the T-ideal
$\CI_{0} \otimes F$, using Theorem~\ref{Shir} to show that this part of
the T-ideal is Noetherian, and then modding it out and applying
Noetherian induction. Unfortunately, working over $C$ might
involve $C$-torsion which could collapse infinite chains when passing to the field of fractions, $F$.
We can use Proposition~\ref{subdir1} to eliminate torsion involved
with a given finite set of elements of $C$, so our strategy is to
show how the whole process just described can be achieved over a
localization of $C$ by a finite number of elements, which are
found independently of the specific chain $\mathcal C$. Thus, we
can work over this localization just as well as over $F$, and pass
back to $C$ by means of Proposition~\ref{subdir1}.

 {\it Step 2.}
We rely heavily on Proposition~\ref{subdir1} in order to eliminate
torsion involved with a given finite set of elements of $C$, with
the aim of modifying $A$ in order to make it more compatible with
$\Gamma_1$. We say that a T-ideal of $F\{ x\}$ is
$C$-\textbf{expanded} if it is generated by polynomials $\subset
C\{ x\}$.  We extend $f_1 = f$ to a set $\{ f_1, \dots, f_k\}
\subset C\{ x\}$ generating a maximal possible $C$-expanded
T-ideal of $F\{ x\}$ contained in $\CI \otimes F$. (Such a finite set
exists since we already have solved Specht's problem over fields,
implying $F\{ x\}$ satisfies the ACC on $C$-expanded T-ideals.)

The coefficients of $f_1, \dots, f_k$ 
involve only finitely many elements of $C$. Utilizing
Proposition~\ref{subdir1}, we can localize at these primes to
obtain a new base ring $C'$ and assume that $A$ has no torsion at
the coefficients of the polynomials $f_1, \dots, f_k$. Let $\CI'$
(resp.~$\CI' \otimes F$) denote the T-ideal of $C'\{x\}$ (resp.~of
$F\{x\}$) generated by $f_1, \dots, f_k$, whose full quiver over
$F$ is denoted as $\Gamma_2$. This might increase our
$C'$-expanded T-ideal over $F$, requiring us to adjoin more
polynomials, and thereby forcing us to localize by finitely many
more primes, but the process must stop since $F\{ x\}$ satisfies
the ACC on $C$-expanded T-ideals. This achieves our goal of
matching a T-ideal  over $C$ with a T-ideal over $F$.

 {\it Step 3.} Our next goal is to hike to a $\bar q$-characteristic coefficient-absorbing polynomial.
  As in Lemma~\ref{induction}, we
take a maximal path in the full quiver $\Gamma_2$. Its polynomial
can be hiked to a finite set of polynomials $\tilde f_1, \dots,
\tilde f_m$. Unfortunately these might involve torsion with new
primes of $C'$. But the torsion over $C'$ in localizing these
finitely many polynomials involves only finitely many prime
elements in $C'$, and by localizing at them we obtain a new base Noetherian ring $C''$ and an algebra
$A'' = A \otimes C''$ over it. Now we can
appeal again to Proposition~\ref{subdir1} and replace $A$ by
$A''$; thereby, we may assume that~$A''$ is $z$-torsion free for
the finitely many primes $z$ at which we localized. (Perhaps $F\{x\}$ has more
$C''$-expanded T-ideals  than $C'$-expanded T-ideals, so we must
return to Step 2 and then Step 3, but this loop must terminate
since $F\{ x\}$ satisfies the ACC on T-ideals.)

In this way, we avoid all
torsion in computing the $\bar q$-characteristic coefficients in
the maximal matrix components, and thereby perform these
calculations in $A''$. In other words, we can use $\tilde f_1,
\dots, \tilde f_m$ (taken over $C''$) to calculate $\bar
q$-characteristic coefficients of the products of the generators
of $A$ in terms of polynomials.

Starting with $C''$ 
we let $\CI'' $ (resp.~$\CI'' \otimes F$) be the T-ideal of $C''\{x\}$
(resp.~of~ $F\{x\}$) generated by $f_1, \dots, f_k$ and $\tilde
f_1, \dots, \tilde f_m$.

 {\it Step 4.} This is the most delicate part of the proof.
Our strategy in this case is to go back to mimic the proof of the
field-theoretic case (Theorem~\ref{Spechtfin}), removing
$C$-torsion step by step when we pass back from $F\{x\}$ to
$C''\{x\}$.  But we must be careful to do everything in a finite
number of steps. We would like to appeal to compactness from
logic, but the argument is more subtle, since certain steps cannot
be put in quantitative form. In particular, we note that the chain
$$\{ \CI_j\otimes_C F :j \in \N \}$$ stabilizes at $\CI_{j_0}\otimes _C F$ for some $j_0$, which we take to be $j_0 = 1$, and we
 define  $\CI_{1;F}^{(0)}\subseteq \CI_1\otimes _C F$
to be the T-ideal of   $A_F$ generated by symmetrized
$\bar{q}$-characteristic coefficient-absorbing polynomials of
$\CI_1 \otimes _C F$ having a non-zero specialization with maximal
degree vector, as described in the proof of
Theorem~\ref{Spechtfin}.

 This is generated by finitely many
polynomials of $A_F$, which can be taken from $A''$ and define a
T-ideal of $A''$ which we call  $\CI_{1}^{(0)}$. Working with
$\CI_{1;F}^{(0)}$ enables us to define finitely many
characteristic coefficients which we define in terms of
polynomials which we now call $g_1, \dots, g_m \in C''\{ x\}.$
Inverting the torsion, i.e., localizing at some $z \in C'',$ we now
may assume that the $g_i$ are $C''$-torsion free (and nonzero since
they localize to nonzero elements of $F\{ x\}$).

We would like to use $g_1, \dots, g_m$ to define ``characteristic
coefficients'' for the elements of  $\CI_{1}^{(0)}\subset A'',$
but unfortunately these are no longer central. But inverting the
$C''$-torsion of the $g_ig_j -g_jg_i$, $1 \le i,j \le m,$ we may
assume that $g_1, \dots, g_m$ all commute, and $\CI_{1}^{(0)}$ is
a module over the commutative Noetherian ring $\hat C : =
C''[g_1, \dots, g_m].$ This is enough for us to apply
Theorem~\ref{Shir} to show that $\CI_{1}^{(0)}$ is a finite
module, and thus Noetherian. (We can define the $\delta$-operators
via Remark~\ref{pat}, together with the module $M$ which is the
T-ideal generated by $\tilde f$. Note that since we only need
consider monomials up to a certain length, we need to adjoin only
finitely many characteristic coefficients, again via localization
and Proposition~\ref{subdir1}.)

 In this way, after
localizing  by finitely many elements of $C$, we pass to finite
modules over Noetherian rings. 

After all of these localizations we have a new affine base ring $
C''' \supset C''$, and work over $\hat C ''' := C '''[g_1,
\dots, g_m]$.  We let $\CI''' $ be the T-ideal of $C'''\{x\}$
generated by the  new polynomials involved with these extra steps. 
Thus $\CI''' \otimes F$ (resp. $\CI''' \otimes {\hat C}$) is the T-ideal of $F\{x\}$ (resp.~of~  $\hat C''' \{x\}$) generated by the same new polynomials.

If $\CI' \otimes F \subset \CI''' \otimes F$, then the quiver of the relatively
free algebra $F\{ x\}/(\CI''' \otimes F)$ is a reduction of~$\Gamma_1$, so
we conclude by induction  on the complexity of the quiver, in view
of Lemma~\ref{induc}.

Thus, we may assume that  $\CI''' \otimes F = \CI' \otimes F$. Next we look at
$(\CI''' \otimes {\hat C})/(\CI' \otimes \hat C)$. By assumption, this a
torsion submodule of the Noetherian module $\CI_{1;F}^{(0)}/(\CI' \otimes \hat C)$ and thus is finite, so if nonzero is annihilated by some nonzero element $z \in C$.
We can remove $z$-torsion one final time (again via
localization and Proposition~\ref{subdir1}), passing to a new base ring $ C'''' \supset C'''$ and T-ideal $\CI''''$ 
(resp.~$\CI'''' \otimes F$)
of $C''''\{x\}$ (resp.~of~ $F\{x\}$). If $\CI' \otimes F\subset \CI'''' \otimes F$, then the quiver of the relatively free algebra $F\{x\}/(\CI'''' \otimes F)$ is a reduction of $\Gamma_1$, so we conclude by
induction  on the complexity of the quiver, in view of
Lemma~\ref{induc}. Thus we may assume that  $\CI'''' \otimes F= \CI' \otimes F$, and since $z$ has been inverted in the localization we
conclude that our ascending chain of T-ideals from $\CI'$
lifts to an ascending chain of T-ideals from $\CI''''$ and we
are done by the process given in the proof of
Theorem~\ref{Spechtfin}. (The point is that the argument of
modding out a certain Noetherian submodule of each T-ideal in
$A \otimes _C C''''$ is algorithmic, depending on
computations involving a finite number of polynomials whose
$C$-torsion we have removed.) This concludes the second proof of
Theorem~\ref{SpechtNoeth}.

\medskip

In summary, we have performed various procedures in order to
enable us to reduce the quiver. These procedures involve a T-ideal
of  $F\{x\}$ which might increase because of the procedure, but
must eventually terminate because  $F\{ x\}$ satisfies the ACC on
T-ideals. But at this stage Step 3 does not vitiate Step 2, and we
can conclude the proof using Step 4 to carve out representable
T-ideals over $C''''$.

Alternatively, one could conclude by applying the compactness in
logic to the proof of Kemer's theorem and checking that we only
need finitely many elements, which can be computed. Fuller details
of the compactness argument are forthcoming when we consider
representability and the universal algebra version of
Theorem~\ref{SpechtNoeth}.

\section{The case where the T-ideals are not necessarily
PI-proper}\label{notprop}

Using the same ideas, we can extend Theorem~\ref{SpechtNoeth}
still further, considering the general case where the T-ideals are
not PI-proper; in other words, the ideals of $C$ generated by the
coefficients of the polynomials in the T-ideals of $C\{ X\}$ do
not contain the element 1. Towards this end, given a set $S$ of
polynomials in $C\{ X\}$, define its \textbf{coefficient ideal} to
be the ideal of $C$ generated by the coefficients of the
polynomials in $S$. We need a few observations about the
multilinearization procedure.

\begin{lem}\label{SpechtN1} If a T-ideal $\CI$ contains a polynomial $f$ with coefficient $c,$
then $\CI$ also contains a multilinear polynomial with coefficient
$c.$
\end{lem}
\begin{proof} First we note that one of the blended components of
$f$ has coefficient $c,$ and then we multilinearize it.
\end{proof}

\begin{lem}\label{SpechtN2} If $c$ is in the coefficient ideal of a T-ideal $\CI$ of
$C\{x\},$ then some multilinear $f\in \CI$ has coefficient $c$.
\end{lem}
\begin{proof} If $c = \sum c_i s_i$ where $s_i$ appears as a
coefficient of $f_i \in \CI$, then taking the $f_i = f_i(x_1,\dots
x_{m_i})$ to be multilinear, we may assume that $c_i$ is the
coefficient of $x_1\cdots x_{m_i}$. Taking $m > \max\{ m_i \}$,
 we see that the coefficient of $x_1\cdots x_{m}$ in
 $$\sum _i s_i f_i(x_1,\dots
x_{m_i})x_{m_i+1}\cdots x_{m}$$ is 1.
\end{proof}

\begin{cor}\label{SpechtN3} A T-ideal is PI-proper iff its coefficient ideal
contains 1. \end{cor}

\begin{prop}\label{SpechtNoet1} Suppose $C$ is a Noetherian integral domain, and $\CI$ is a T-ideal with
coefficient ideal $I$. Then there is a polynomial $f \in \Z \{x\}$
for which $cf \in \CI$ for all $c\in I.$
\end{prop}
\begin{proof} Since $C$ is Noetherian, we can write $I =
\sum_{i=1}^t Cc_i,$ and then it is enough to prove the assertion
for $c = c_i$, $1 \le i \le t.$

We take the relatively free, countably generated algebra $A$ whose
generators $\{ y_1, y_2 \dots, \}$ are given the lexicographic
order, and let $M_m$ denote the space of multilinear words of
degree $m$ in $\{ y_1, \dots, y_m \}$. In view of Shirshov's
Height Theorem \cite[Theorem~2.3]{BR}, the space $\sum _i c_i M_m$
has bounded rank as a $\Z$-module. On the other hand, there is a
well-known action of the symmetric group $S_m$ acting on the
indices of $y_1, \dots, y_m$ described in \cite[Chapter~5]{BR}. In
particular, \cite[Theorem~5.51]{BR} gives us a rectangle such that
any multilinear polynomial $f$ whose Young diagram contains this
rectangle satisfies $c_i f \in \CI.$
\end{proof}

Write $\mathcal M(K)$ for the generic $n\times n$ matrix algebra
with characteristic coefficients over a commutative ring~$K$.
 Zubkov~\cite{Z} proved that the canonical map $\mathcal M(\Z) \to \mathcal M(\Z/p\Z)$ has kernel
$p \mathcal M(\Z) .$ (As noted above, this is false if one does not adjoin
characteristic coefficients.) He also proved that the Hilbert series of the
algebra $A\otimes _\Z \Q$ over $\Q$ and of $A\otimes _\Z \Z/p\Z$
over $\Z /p\Z$ coincide, implying $\CI$ is a free $\Z$-module.

\begin{cor}\label{SpechtNoet2} If $\CI$ is a T-ideal with
coefficient ideal $I$, there is a PI-proper T-ideal of $C\{x\}$
whose intersection with $I\{x\}$ is contained in $\CI.$
\end{cor}\begin{proof} We need to show that
\begin{equation}\label{ver1} \CI \cap c A = c\CI\end{equation} for any $c \in C.$
We take the polynomial $f$ of Proposition~\ref{SpechtNoet1}. In
view of Proposition~\ref{Lewin}, the T-ideal of $f$ contains the
set of identities of some finite dimensional algebra, and thus of
$M_n(C)$ for some $n$. Adjoining characteristic coefficients, we
may replace $\CI$ by a T-ideal of the free algebra with
characteristic coefficients, and conclude with Zubkov's results
~\cite{Z} quoted above.
\end{proof}

\begin{thm}\label{SpechtNoeth1} Any T-ideal in the free algebra $C\{ x\}$ is finitely based,
for any commutative Noetherian ring $C$. \end{thm}
\begin{proof} 
By Noetherian induction, we may assume that the theorem holds over
$C/I$ for every nonzero ideal $I$ of $C$. Thus, by
Remark~\ref{Spechtobs2}, $C$ is an integral domain. If $C$ is
finite, then it is a field, and we are done by
Theorem~\ref{Spechtfin}. So we may assume that $C$ is an infinite
integral domain. We need to show that any T-ideal generated by a
given set of polynomials $\{g_1, g_2, \dots \}$ is finitely based.
The coefficient ideals of $\{g_1, g_2, \dots g_j\}$ stabilize to
some ideal $I$ of $C$ at some $j_0$, since $C$ is Noetherian. We
 let $A_0$ denote the relatively free algebra with
respect to the T-ideal generated by $g_1, \dots, g_{j_0},$
Inductively, we let $A_i$ denote the relatively free algebra with
respect to the T-ideal generated by $f_{j_0+1}, \dots, f_{j_0+i},$
and take a PI-proper polynomial $f_{i+1}$, not in $\id(A_i)$ such
that $cf_{i+1}$ is in the T-ideal generated by $g_{i+1}$ in $A_i$
for all $c$ in the coefficient ideal of~$g_{i+1}$. (Such a
polynomial exists in view of Proposition~\ref{SpechtNoet1}.) This
gives us an ascending chain of PI-proper T-ideals of $A_0$, which
must terminate in view of Theorem~\ref{SpechtNoeth}, a
contradiction.
\end{proof}

\subsubsection{Digression: Consequences of torsion for relatively free
algebras}

Torsion has been so useful in this paper that we collect a few
more elementary properties and apply them to relatively free
algebras.

\begin{lem}\label{tor4} Suppose $C$ is a Noetherian integral domain, and $A$ is a relatively free affine
$C$-algebra.
\begin{enumerate} \item $A$ has $p$-torsion for only finitely many
prime numbers $p$. \item There is some $k_0$ such that
$p^{k}\operatorname{-}\!\tor(A) =
p^{k+1}\operatorname{-}\!\tor(A)$ for all $k> k_0$ and all prime
numbers $p$. \item Let $\phi_k \co A \to A \otimes \Z/p^k\Z$ denote
the natural homomorphism. If $p^k A \ne p^{k+1} A,$ then $\ker
\phi_k \ne \ker \phi_{k+1}.$
\end{enumerate}
\end{lem}
\begin{proof} $p^k\operatorname{-}\!\tor(A)$ is a T-ideal for each $k$.
Let $\CI_k$ be the T-ideal generated by $p^k$-torsion elements.
The $\CI_k$ stabilize for some $k_0$, yielding (2), and (3)
follows since once the chain stabilizes we have $p^k A = p^{k+1}
A$. Likewise, the direct sum of these T-ideals taken over all
primes stabilizes, yielding (1).
\end{proof}

\subsection{Applications to relatively free algebras}

As Kemer~\cite{Kem2} noted, the ACC on T-ideals formally yields a
Noetherian-type theory. We apply this method to
Theorem~\ref{SpechtNoeth1}.

\begin{prop}\label{Trad} Any relatively free algebra $A$ over a commutative Noetherian ring
has a unique maximal nilpotent T-ideal $N(A)$.\end{prop}
\begin{proof} By ACC, there is a maximal nilpotent T-ideal, which is unique since
the sum of two nilpotent T-ideals is a nilpotent T-ideal.
\end{proof}

\begin{defn}\label{Tpr} The ideal $N(A)$ of Proposition~\ref{Trad} is called the \textbf{T-radical}. An algebra $A$ is
\textbf{T-prime} if the product of two nonzero T-ideals is
nonzero. A T-ideal $\CI$ of $A$ is \textbf{T-prime} if $A/\CI$ is
a T-prime algebra.
\end{defn}

\begin{prop}\label{struc1} The T-radical
 is the intersection of a finite number of T-prime T-ideals.\end{prop}
\begin{proof} Each T-prime T-ideal contains the T-radical, which we thus can mod out. Then just copy the usual argument using Noetherian induction.
\end{proof}

 Kemer characterized all T-prime algebras of characteristic 0,
cf.~\cite[Theorem~6.64]{BR}. The situation in nonzero
characteristic is much more difficult, but in general we can
reduce to the field case.

\begin{prop} Each T-prime, relatively free algebra $A$ with 1 over a commutative Noetherian
ring $C$ is either the free $C$-algebra or is PI-equivalent to a
relatively free algebra over a field. In particular, either $A$ is
free or PI.
\end{prop}\begin{proof} The center $Z$ of $A$ is an integral domain over
which $A$ is torsionfree, since if $c \in C$ has torsion then $0
= (cA)\Ann _A(c)$ implies $cA = 0$ so $c = 0.$ If $Z$ is finite
then it is a field and we are done. If $Z$ is infinite then $A$ is
PI-equivalent to $A \otimes Z K$ where $K$ is the field of
fractions of $Z$.
\end{proof}

 In this we we see that this theory, in particular Corollary~\ref{SpechtNoet2},
provides a method for generalizing results about relatively free
PI-algebras to relatively free algebras in a variety which is not
necessarily PI-proper. For example, let us generalize a celebrated
theorem of Braun\cite{Br}-Kemer-Razmyslov:

\begin{thm}\label{Jnilp} The Jacobson radical $J$ of any relatively free affine algebra $A$ is
nilpotent. \end{thm}
\begin{proof} Modding out the T-radical, and applying
Proposition~\ref{struc1}, we may assume that $A$ is T-prime. If it
is free then $J = 0$, so we may assume that $A$ is PI, where we are
done by Braun's theorem.
\end{proof}

Of course there is no hope to generalize this result to
non-relatively free algebras, since the nilradical of an affine
algebra need not be nilpotent.

\end{document}